\documentclass[11pt]{article}

\usepackage[T1]{fontenc}
\usepackage[utf8]{inputenc}
\usepackage[margin=1in]{geometry}
\usepackage{setspace}
\IfFileExists{newtxtext.sty}{\usepackage{newtxtext}\usepackage{newtxmath}}{}

\usepackage{amsmath}
\usepackage{amssymb}
\usepackage{amsfonts}

\usepackage{amsthm}
\usepackage{mathtools}
\usepackage{bm}
\IfFileExists{bbm.sty}{\usepackage{bbm}}{}

\usepackage{graphicx}
\usepackage{subcaption}
\usepackage{booktabs}
\usepackage{tabularx}
\usepackage{multirow}
\usepackage{array}

\usepackage{enumitem}
\IfFileExists{algorithm.sty}{\usepackage{algorithm}\usepackage{algorithmic}}{}

\usepackage{tikz}
\usetikzlibrary{arrows.meta, positioning, shadows, fit, backgrounds, calc, decorations.markings, patterns}

\usepackage{authblk}

\usepackage[round,compress]{natbib}
\usepackage[colorlinks=true,allcolors=green!60!black]{hyperref}

\usepackage[framemethod=TikZ]{mdframed}
\usepackage{needspace}
\mdfdefinestyle{principleframe}{%
  linewidth=0.6pt,
  linecolor=black,
  backgroundcolor=white,
  innertopmargin=8pt,
  innerbottommargin=8pt,
  innerleftmargin=10pt,
  innerrightmargin=10pt,
  skipabove=\medskipamount,
  skipbelow=\medskipamount,
  frametitlefont=\bfseries,
  theoremseparator={.\ },
  nobreak=true,
}

\newcommand{\defeq}{\mathrel{\mathop:}=}

\newcommand{\E}{\mathbb{E}}

\newcommand{\KL}{\mathrm{KL}}

\newcommand{\R}{\mathbb{R}}
\newcommand{\RR}{\mathbb{R}}

\newcommand{\calX}{\mathcal{X}}
\newcommand{\calA}{\mathcal{A}}
\newcommand{\calF}{\mathcal{F}}
\newcommand{\calP}{\mathcal{P}}

\newcommand{\calN}{\mathcal{N}}

\newcommand{\rhotail}{\rho}  

\newcommand{\diff}{\,\mathrm{d}}
\newcommand{\1}{\mathbf{1}}
\newcommand{\PP}{\mathbb{P}}
\newcommand{\EE}{\mathbb{E}}

\theoremstyle{plain}
\newtheorem{theorem}{Theorem}[section]
\newtheorem{proposition}[theorem]{Proposition}
\newtheorem{lemma}[theorem]{Lemma}
\newtheorem{corollary}[theorem]{Corollary}

\theoremstyle{definition}
\newtheorem{definition}[theorem]{Definition}
\newtheorem{assumption}[theorem]{Assumption}

\newtheorem{principle}[theorem]{Principle}

\theoremstyle{remark}
\newtheorem{remark}[theorem]{Remark}

\title{\textbf{Bayes Risk for Goodness of Fit Tests}}
\author[2]{Nicholas G. Polson}
\author[3]{Vadim Sokolov}
\author[4]{Daniel Zantedeschi}

\affil[2]{Booth School of Business, University of Chicago, Chicago, IL 60637\\
\texttt{ngp@chicagobooth.edu}}
\affil[3]{Volgenau School of Engineering, George Mason University, Fairfax, VA 22030\\
\texttt{vsokolov@gmu.edu}}
\affil[4]{School of Information Systems, Muma College of Business, University of South Florida, Tampa, FL 33620\\
\texttt{danielz@usf.edu}}
\date{February 2026}

\begin{document}
\maketitle

\begin{abstract}
We develop a unified framework for goodness-of-fit (GOF) testing through the lens of Bayes risk. Classical GOF procedures are commonly calibrated either at fixed significance level (CLT scale) or through exponential error exponents (LDP scale). We establish that Bayes-risk optimal calibration operates on the moderate-deviation (MDP) scale, producing canonical $\sqrt{\log n}$ inflation of rejection thresholds and polynomially decaying Type~I error.

Our main contributions are:
(i)~we formalise the Rubin--Sethuraman program for KS-type statistics as a risk-calibration theorem with explicit regularity conditions on priors and empirical-process functionals;
(ii)~we develop the precise connection between Bayes-risk expansions and Sanov information asymptotics, showing how $\log n$-order truncations arise naturally when risk, rather than pure exponents, is the evaluation criterion;
(iii)~we provide detailed applications to location testing under Laplace families, shape testing via Bayes factors, and connections to Fisher information geometry.
The organizing principle throughout is that sample size enters Bayes-optimal GOF cutoffs through the MDP scale, unifying KS-based and Sanov-based perspectives under a single risk criterion.
\end{abstract}


\newpage

\section{Introduction}
\label{sec:intro}

Goodness-of-fit (GOF) testing occupies a central position in statistical methodology, serving as the primary mechanism for model criticism, specification search, and validation of probabilistic assumptions.
Classical GOF procedures, including Kolmogorov--Smirnov (KS) statistics, likelihood-ratio tests, chi-squared tests, and entropy-based criteria, have been studied extensively under two dominant asymptotic frameworks:
\emph{fixed-$\alpha$ calibration} (CLT scale), where critical values control Type~I error at a constant level, and
\emph{exponential error exponents} (LDP scale), where performance is characterised by large-deviation rate functions and Chernoff bounds \citep{Chernoff1952, DemboZeitouni1998}.

Neither framework is appropriate when the inferential objective is \emph{Bayes risk minimisation} \citep[cf.][]{Berger1985, LeCam1986}.
Fixed-$\alpha$ calibration ignores the fact that, as sample size grows, an increasing number of alternatives become statistically distinguishable and must be jointly serviced.
Pure large-deviation calibration enforces overly stringent rejection thresholds that sacrifice power against local and moderate alternatives, leading to suboptimal Bayesian risk even when Type~I error is extremely small.

The central message of this paper is that Bayesian calibration of GOF tests naturally operates on the \emph{moderate-deviation} (MDP) scale.
This intermediate regime, lying strictly between CLT and LDP scales, arises because the number of alternatives that a sample of size~$n$ can distinguish from the null grows polynomially, and a Bayes-risk criterion must service all of them.
Concretely, at sample size~$n$ the experiment resolves distributions at Fisher distance $\sqrt{\log n/n}$ from the null; any prior that is absolutely continuous near $\calF_0$ places polynomial mass in this shell.
Balancing false rejections against the cumulative missed-detection cost over this growing set of distinguishable alternatives forces thresholds to inflate as $\sqrt{\log n}$, Type~I error to decay polynomially, and Bayes risk to concentrate on the boundary shell.

The information-theoretic lineage of this perspective is extensive.
\citet{Sanov1957} established that empirical measures satisfy a large-deviation principle with Kullback--Leibler rate function.
\citet{Bahadur1960} and \citet{BahadurRao1960} showed that polynomial prefactors in tail probabilities affect statistical efficiency at moderate scales.
\citet{RubinSethuraman1965a, RubinSethuraman1965b} analysed Bayes risk for empirical-process statistics and derived moderate-deviation optimality for KS-type procedures.
\citet{Good1955} emphasised weight-of-evidence interpretations of log-likelihood ratios, while Pearson's $\chi^2$ test and entropy-deficit criteria \citep{Jaynes1957} were later shown to share identical large-deviation rate functions.
What has been missing is a \emph{single generative mechanism} that explains why these distinct threads converge at the MDP scale, and a template (Lemma~\ref{lem:risk-template}) that makes each instantiation a two-line corollary.

Our main contributions are as follows.
We state an explicit \emph{Bayes-Risk MDP Principle} (Section~\ref{sec:setup}) and prove it as a theorem under regularity conditions that cover KS-type statistics, Sanov-based tests, and standard parametric GOF problems.
We formalise the Rubin--Sethuraman program for KS statistics (Section~\ref{sec:ks}) as a risk-calibration theorem with sharp constants, showing that the Bayes-optimal threshold satisfies $t_n^* = \sqrt{(\kappa/4)\log n} + O(\sqrt{\log\log n})$.
We develop the connection between Bayes risk and Sanov information asymptotics (Section~\ref{sec:sanov}), showing that the effective KL exponent is truncated to $O(\log n / n)$ under Bayes-risk optimisation, and providing the information-theoretic interpretation of MDP scaling.
Finally, we apply the theory to location testing under Laplace families, shape testing via Bayes factors, multinomial GOF, and Fisher information geometry (Sections~\ref{sec:apps}--\ref{sec:fisher}).
Rubin--Sethuraman identify the MDP calibration in the KS setting; our contribution is to isolate the underlying $(\rhotail,\kappa)$ mechanism and show it applies uniformly across GOF statistics, including empirical-process, KL/Sanov, multinomial, and parametric geometric instances.

Recent work \citep{DPSZ2026} develops a unified theory of Bayesian hypothesis testing via moderate deviation asymptotics, focusing on Bayes factors and integrated risk expansions.
While both approaches identify the moderate deviation regime as central for Bayesian calibration, the present paper isolates a distinct mechanism: a generic risk decomposition template based solely on (i)~Gaussian-type null tails and (ii)~local prior mass exponents.
Our results do not depend on Bayes factors or likelihood-ratio structure per se.
Instead, Theorem~\ref{thm:mdp-generic} and Lemma~\ref{lem:risk-template} provide a reusable calibration principle applicable to a broad class of statistics, including goodness-of-fit tests, empirical processes, and divergence-based procedures.
In this sense, the moderate deviation scaling emerges here as a structural consequence of risk balancing, rather than as a property specific to Bayesian hypothesis testing.
Our references to Dawid's decomposition and prequential results follow the published accounts in \citet{Dawid2011PosteriorModelProbabilities,Dawid1992Prequential}.

The paper is organised as follows.
Section~\ref{sec:setup} formalises the GOF testing problem, defines Bayes risk, and states the organizing principle.
Section~\ref{sec:ks} establishes Bayes-risk optimality of MDP calibration for KS-type statistics.
Section~\ref{sec:sanov} develops the Sanov/information-asymptotic perspective.
Section~\ref{sec:apps} gives applications.
Section~\ref{sec:fisher} connects to Fisher information geometry.
Section~\ref{sec:numerics} provides numerical verification of the calibration template.
Section~\ref{sec:discussion} discusses implications and extensions.
Appendix~\ref{app:triangulation} provides a compact triangulation linking Bayes factors (Good), large deviations (Hoeffding), entropy geometry (Jaynes), and Wilks' chi-square limit through KL curvature, clarifying the geometric foundation of the moderate-deviation boundary derived here.

\section{Setup and the Bayes-Risk GOF Principle}
\label{sec:setup}

\subsection{The Testing Problem}

Let $X_1,\ldots,X_n$ be i.i.d.\ with distribution $F$ on a measurable space $(\calX,\calA)$. The goodness-of-fit testing problem is
\begin{equation}\label{eq:testing-problem}
H_0: F \in \calF_0
\qquad \text{vs.} \qquad
H_1: F \in \calF_1,
\end{equation}
where $\calF_0$ and $\calF_1$ are disjoint classes of distributions.
We allow several standard configurations: a simple null $\calF_0=\{F_0\}$; a parametric null $\calF_0=\{F_\theta:\theta\in\Theta_0\}$ \citep[see][for consistency under growing parameter spaces]{KieferWolfowitz1956}; and composite or nonparametric alternatives.
A (possibly randomized) test is a measurable map $\delta_n:\calX^n\to[0,1]$, where $\delta_n(x^n)$ is the probability of rejecting $H_0$ given data $x^n$.

\subsection{Empirical Measures and GOF Statistics}\label{subsec:empirical}

Many GOF procedures are functionals of the empirical measure
\begin{equation}\label{eq:empirical-measure}
\hat{\pi}_n \;\defeq\; \frac{1}{n}\sum_{i=1}^n \delta_{X_i},
\end{equation}
or equivalently the empirical distribution function $F_n(t) = n^{-1}\sum_{i=1}^n \1\{X_i\le t\}$.
GOF statistics can be written as $T_n=\mathsf{T}(\hat{\pi}_n;F_0)$ for a discrepancy functional $\mathsf{T}$.
Representative examples include the Kolmogorov--Smirnov statistic $S_n=\sup_t |F_n(t)-F_0(t)|$, Cram\'er--von Mises, Anderson--Darling, and likelihood-ratio statistics.

\subsection{Bayes Risk Formulation}

We embed GOF testing in Bayesian decision theory by placing priors on $\calF_0$ and $\calF_1$ and assigning costs to Type~I and Type~II errors.

\begin{definition}[Bayes risk for testing]\label{def:bayes-risk}
Let $\pi_0,\pi_1>0$ with $\pi_0+\pi_1=1$ denote prior probabilities on the hypotheses.
Let $\Pi_0$ be a probability measure on $\calF_0$ and $\Pi_1$ a probability measure on $\calF_1$.
Let $L_0,L_1>0$ denote the losses for Type~I and Type~II errors.
The \emph{Bayes risk} of a test $\delta_n$ is
\begin{equation}\label{eq:bayes-risk}
B_n(\delta_n)
=
\pi_0 L_0\, \E_{F\sim \Pi_0}\!\big[\E_F[\delta_n(X^n)]\big]
\;+\;
\pi_1 L_1\, \E_{F\sim \Pi_1}\!\big[\E_F[1-\delta_n(X^n)]\big].
\end{equation}
\end{definition}

\begin{remark}[Error probabilities and averaged errors]\label{rem:avg-errors}
For any $F$, define the rejection probability $\alpha_n(F)\defeq\E_F[\delta_n(X^n)]$.
The prior-averaged errors are $\bar\alpha_n \defeq \E_{F\sim\Pi_0}[\alpha_n(F)]$ and $\bar\beta_n \defeq \E_{F\sim\Pi_1}[1-\alpha_n(F)]$,
so that $B_n(\delta_n)=\pi_0 L_0\,\bar\alpha_n+\pi_1 L_1\,\bar\beta_n$.
Bayes calibration balances Type~I and Type~II errors after averaging over the null and alternative classes.
\end{remark}

\subsection{Reduction to a Likelihood-Ratio Form}\label{subsec:lr-reduction}

Define the prior predictive (mixture) measures on $\calX^n$:
\[
M_0^{(n)}(\cdot)\;\defeq\;\int P_F^{\otimes n}(\cdot)\,\Pi_0(dF),
\qquad
M_1^{(n)}(\cdot)\;\defeq\;\int P_F^{\otimes n}(\cdot)\,\Pi_1(dF).
\]

\begin{proposition}[Bayes-optimal test]\label{prop:bayes-np}
Assume $M_0^{(n)}$ and $M_1^{(n)}$ are mutually absolutely continuous.
Then any Bayes-risk minimiser can be taken to be a threshold rule in the predictive likelihood ratio $\Lambda_n(X^n)\defeq dM_1^{(n)}/dM_0^{(n)}(X^n)$:
\begin{equation}\label{eq:bayes-optimal-test}
\delta_n^*(X^n)
=
\1\!\left\{\Lambda_n(X^n)>\tau\right\}
\;+\;
\gamma\,\1\!\left\{\Lambda_n(X^n)=\tau\right\},
\qquad
\tau\;\defeq\;\frac{\pi_0 L_0}{\pi_1 L_1},
\end{equation}
for some $\gamma\in[0,1]$.
\end{proposition}

\begin{proof}
The data-then-posterior form \eqref{eq:fubini2} shows $B_n(\delta_n)=\int\bigl[\int L(F,\delta_n)\,\Pi(dF\mid x^n)\bigr]\,m(x^n)\,dx^n$.
For 0--1 loss with costs $L_0,L_1$, the inner integral is minimised pointwise by choosing $\delta_n(x^n)=1$ iff $\pi_1 L_1\,\PP(F\in\calF_1\mid x^n)>\pi_0 L_0\,\PP(F\in\calF_0\mid x^n)$.
By Bayes' rule, $\PP(F\in\calF_i\mid x^n)\propto \pi_i\,dM_i^{(n)}/dx^n$, so the decision reduces to $dM_1^{(n)}/dM_0^{(n)}>\pi_0 L_0/(\pi_1 L_1)$.
This is the classical Bayesian binary hypothesis testing result \citep[Ch.~4]{Berger1985}.
\end{proof}

\subsection{Fubini Decompositions}\label{subsec:fubini}

A key structural device is the ability to rewrite Bayes risk by swapping the order of integration between the data space and the model space.
Let $\Pi=\pi_0\Pi_0+\pi_1\Pi_1$ denote the mixture prior, $p(x^n\mid F)$ a density version, and $m(x^n) \defeq \int p(x^n\mid F)\,\Pi(dF)$ the marginal.

The \emph{prior-then-data form} writes Bayes risk as an expectation of loss under the joint:
\begin{equation}\label{eq:fubini1}
B_n(\delta_n)
=
\int_{\calF_0\cup\calF_1}
\int_{\calX^n}
L\!\big(F,\delta_n(x^n)\big)\, p(x^n\mid F)\, dx^n\, \Pi(dF).
\end{equation}
By Fubini, the \emph{data-then-posterior form} is:
\begin{equation}\label{eq:fubini2}
B_n(\delta_n)
=
\int_{\calX^n}
\left[
\int_{\calF_0\cup\calF_1}
L\!\big(F,\delta_n(x^n)\big)\,\Pi(dF\mid x^n)
\right]
m(x^n)\, dx^n.
\end{equation}
Equation~\eqref{eq:fubini1} is useful for understanding how the alternative space contributes to risk; \eqref{eq:fubini2} is the operational form underlying the proof of Proposition~\ref{prop:bayes-np}.

\subsection{Three Calibration Regimes}\label{subsec:three-regimes}

Calibration of a GOF test amounts to choosing a rejection threshold $t_n$ for a statistic $T_n$.
Write $\alpha_n \asymp \PP_{F_0}(T_n>t_n)$.

\begin{definition}[Calibration regimes]\label{def:regimes}
The canonical regimes are:
the \emph{CLT regime}, where $\alpha_n=\alpha\in(0,1)$ is fixed and $t_n=O(1)$ for standardised statistics;
the \emph{MDP regime}, where $\alpha_n=n^{-c}$ for some $c>0$ and $t_n=O(\sqrt{\log n})$ for Gaussianised statistics; and
the \emph{LDP regime}, where $\alpha_n=\exp(-cn)$ for some $c>0$ and $t_n=O(\sqrt{n})$ for Gaussianised statistics.
\end{definition}

\subsection{The Organizing Principle}\label{subsec:principle}

The following principle is the structural backbone of this paper. Each main theorem instantiates it in a specific GOF setting.

\needspace{8\baselineskip}%
\begin{mdframed}[style=principleframe]
\begin{principle}[Bayes risk forces the MDP scale]\label{princ:mdp}
Under Bayes-risk optimization for goodness-of-fit testing, the optimal rejection threshold inflates at order $\sqrt{\log n}$ for Gaussianised statistics, because risk trades Type~I polynomial decay against Type~II detection under prior-weighted mixtures.
Specifically, whenever
(i)~the null tail of the GOF statistic is sub-Gaussian at the $\sqrt{n}$-scale, and
(ii)~the alternative prior $\Pi_1$ assigns polynomial mass to shrinking neighbourhoods of $\calF_0$,
the Bayes-optimal threshold satisfies $t_n^*\asymp \sqrt{\log n}$ and the Type~I error decays polynomially: $\alpha_n^* \asymp n^{-c}$ for a problem-dependent constant $c>0$.
\end{principle}
\end{mdframed}

\medskip

The reduction chain that this principle encodes is:
\begin{multline*}
\text{GOF problem}
\;\rightarrow\;
\text{Bayes risk criterion}
\;\rightarrow\;
\text{mixture tail / truncation}\\
\;\rightarrow\;
\text{MDP scaling}
\;\rightarrow\;
\text{canonical thresholds}.
\end{multline*}

\begin{figure}[ht]
\centering
\begin{tikzpicture}[>=Stealth, font=\small]
  \draw[->, thick] (0,0) -- (12.2,0);
  \node[anchor=west] at (12.3,0) {$t_n$};
  \draw[thick] (0.3,-0.12) -- (0.3,0.12);
  \draw[thick] (4.0,-0.12) -- (4.0,0.12);
  \draw[thick] (8.0,-0.12) -- (8.0,0.12);
  \draw[thick] (11.8,-0.12) -- (11.8,0.12);
  \node[anchor=south, font=\small\bfseries] at (2.15,0.25) {CLT};
  \node[anchor=south, font=\small\bfseries] at (6.0,0.25) {MDP};
  \node[anchor=south, font=\small\bfseries] at (9.9,0.25) {LDP};
  \node[anchor=north] at (2.15,-0.25) {$t_n = O(1)$};
  \node[anchor=north] at (6.0,-0.25) {$t_n = \sqrt{a\log n}$};
  \node[anchor=north] at (9.9,-0.25) {$t_n = O(\sqrt{n})$};
  \node[anchor=north, gray] at (2.15,-0.85) {$\alpha_n = \alpha$};
  \node[anchor=north, gray] at (6.0,-0.85) {$\alpha_n \asymp n^{-\kappa/2}$};
  \node[anchor=north, gray] at (9.9,-0.85) {$\alpha_n \asymp e^{-cn}$};
  \node[anchor=north] at (2.15,-1.55) {$B_n \not\to 0$};
  \node[anchor=north, font=\small\bfseries] at (6.0,-1.55) {$B_n^{*} \asymp \bigl(\tfrac{\log n}{n}\bigr)^{\!\kappa/2}$};
  \node[anchor=north] at (9.9,-1.55) {Type~II dominates};
  \node[anchor=east, font=\footnotesize\itshape] at (-0.1,-0.25) {Threshold};
  \node[anchor=east, font=\footnotesize\itshape, gray] at (-0.1,-0.85) {Type~I};
  \node[anchor=east, font=\footnotesize\itshape] at (-0.1,-1.55) {Bayes risk};
  \node[anchor=north, font=\footnotesize] at (6.0,-2.05) {Bayes-risk optimal};
\end{tikzpicture}
\caption{The three calibration regimes.  Only the moderate-deviation (MDP) scale achieves vanishing Bayes risk at the optimal rate.}\label{fig:regime-trichotomy}
\end{figure}
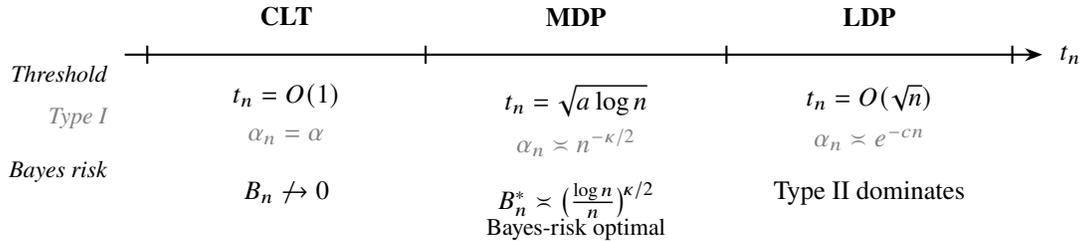

Each main result in this paper follows this chain:
KS/empirical-process functionals (Theorem~\ref{thm:ks-mdp}),
Sanov/LDP truncation (Theorem~\ref{thm:mdp-truncation}),
concrete models, namely Laplace location (Theorem~\ref{thm:laplace-sign-threshold}) and multinomial chi-squared (Theorem~\ref{thm:chi-squared-threshold}), and
Fisher geometry (Theorem~\ref{thm:fisher-mdp}).

\subsection{First Formal Anchor}

We now state the first formal version of the principle.
Different GOF statistics induce different Gaussian tail constants; we normalise them as follows.

\needspace{8\baselineskip}%
\begin{mdframed}[style=principleframe]
\noindent\textbf{Convention (Gaussian tail normalisation).}\quad
Throughout, we parameterise Gaussian-type null tails in the form
\begin{equation}\label{eq:tail-convention}
\PP_{F_0}\!\bigl(\sqrt{n}\,T_n > t\bigr) \;\asymp\; \exp\!\bigl(-2\rhotail\,t^{2}\bigr), \qquad t\to\infty,
\end{equation}
where $\rhotail>0$ is a statistic-specific constant.
The factor~$2$ is purely notational; it matches the Kolmogorov tail $\sim 2e^{-2t^2}$ and could be absorbed into $\rhotail$; we keep it to simplify later specialisations.
For the KS statistic, $\rhotail=1$.
\end{mdframed}

\smallskip\noindent\textbf{Notation.} Throughout, $f\asymp g$ means $c\,g\le f\le C\,g$ for positive constants $c,C$ independent of~$n$.
Where a statistic-specific polylogarithmic prefactor refines the template rate (e.g., the Rubin--Sethuraman shell factor for KS), it is stated explicitly in the specialised theorem.

\begin{theorem}[MDP is Bayes-risk optimal for GOF]\label{thm:mdp-generic}
Suppose:
\begin{enumerate}[label=(\roman*),leftmargin=2em]
\item \textbf{Sub-Gaussian null tail:} the GOF statistic satisfies $\PP_{F_0}(\sqrt{n}\,T_n > t) \asymp \exp\{-2\rhotail\,t^2\}$ for $t\to\infty$;
\item \textbf{Local detectability:} alternatives at distance $\varepsilon$ from $\calF_0$ are detected once $\varepsilon \gg t/\sqrt{n}$;
\item \textbf{Polynomial prior mass:} $\Pi_1\!\bigl(\{F : d(F,\calF_0) \le \varepsilon\}\bigr) \asymp \varepsilon^{\kappa}$ for small $\varepsilon > 0$.
\end{enumerate}
Set $t_n = \sqrt{a\log n}$.  Then the Bayes risk of the threshold test $\delta_{n,t_n}$ satisfies
\begin{equation}\label{eq:generic-risk-decomp}
B_n(t_n) \;\asymp\; n^{-2\rhotail\,a} \;+\; \bigl(a\log n / n\bigr)^{\kappa/2}.
\end{equation}
Minimising over $a$ yields $a^{*} = \kappa/(4\rhotail)$, whence
the Bayes-optimal threshold is $t_n^{*} = \sqrt{\kappa\log n/(4\rhotail)} + O(\sqrt{\log\log n})$,
Type~I error decays polynomially as $\alpha_n^{*} \asymp n^{-\kappa/2}$,
and the optimal Bayes risk is $B_n^{*} \asymp (\log n / n)^{\kappa/2}$.
By contrast, fixed-$\alpha$ calibration ($a=0$) leaves $B_n$ bounded away from zero, and LDP calibration ($a \propto n$) forces Type~II dominance.
\end{theorem}

\begin{proof}
Decompose the Bayes risk as $B_n(t_n) = \pi_0 L_0\,\alpha_n(t_n) + \pi_1 L_1\,\bar\beta_n(t_n)$, absorbing the positive constants into $\asymp$.

\emph{Type~I term.}
Setting $t_n = \sqrt{a\log n}$ in assumption~(i) gives $\alpha_n \asymp \exp\{-2\rhotail\, a\log n\} = n^{-2\rhotail\,a}$.

\emph{Type~II term.}
By assumption~(ii), alternatives at distance $\varepsilon > t_n/\sqrt{n} = \sqrt{a\log n/n}$ from $\calF_0$ are detected with probability tending to one; alternatives closer than this radius contribute the bulk of the missed-detection cost.
Thus $\bar\beta_n \asymp \Pi_1\!\bigl(\{F: d(F,\calF_0) \le \sqrt{a\log n/n}\}\bigr)$.
By assumption~(iii), $\Pi_1(d(F,\calF_0)\le\varepsilon)\asymp \varepsilon^{\kappa}$, so $\bar\beta_n \asymp (a\log n/n)^{\kappa/2}$.

\emph{Optimisation.}
Write $B_n(a) \asymp n^{-2\rhotail\,a} + (a\log n/n)^{\kappa/2}$.
Setting $\partial/\partial a$ of the leading exponents to balance: $2\rhotail\,a = \kappa/2$ yields $a^{*}=\kappa/(4\rhotail)$.
Direct substitution gives $B_n^{*} \asymp (\log n/n)^{\kappa/2}$.
For $a=0$: $B_n \ge \pi_0 L_0 > 0$.
For $a \propto n/\log n$: $\bar\beta_n \asymp 1$, so $B_n$ is bounded below.
Full details for the KS specialisation are in Appendix~\ref{app:proof-ks-mdp}.
\end{proof}

The two terms in \eqref{eq:generic-risk-decomp} encode a universal trade-off: the first is the false-rejection cost driven by the tail-rate parameter~$\rhotail$; the second is the missed-detection cost governed by the prior exponent~$\kappa$.
Under the stated regularity conditions, neither CLT nor LDP calibration can optimise this balance; the moderate-deviation scale is the unique optimiser.
Every subsequent theorem in this paper is an instantiation of \eqref{eq:generic-risk-decomp} with specific values of~$\rhotail$ and~$\kappa$.
The $\asymp$ in \eqref{eq:generic-risk-decomp} captures the polynomial rate in~$n$; individual specialisations may carry statistic-specific polylogarithmic prefactors (e.g., the Rubin--Sethuraman shell factor for KS) that refine the Type~II term without affecting the optimizer~$a^*$.
The triangulation of evidence measures in Appendix~\ref{app:triangulation} provides the geometric foundation for this universal trade-off.

A notable feature of \eqref{eq:generic-risk-decomp} is that the two error terms operate at \emph{different polynomial orders}: the Type~I term $n^{-2\rhotail\,a}$ is controlled by the tail exponent alone, while the Type~II term $(a\log n/n)^{\kappa/2}$ carries an additional polylogarithmic factor through the $a\log n$ numerator.
This separation of orders is what makes the polynomial prefactors decisive for risk minimisation rather than negligible.
\citet{Hoeffding1965} identified the same phenomenon in the multinomial setting: the asymptotically optimal GOF test for discrete distributions depends on polynomial corrections to the Sanov exponential rate, and these corrections govern the relative efficiency of competing test statistics at moderate scales.
In our framework, Hoeffding's polynomial prefactors are precisely the $(\log n)^{\kappa/2}$ terms that enter through the prior mass condition and drive the risk minimum to the MDP scale.

The connection to the Jeffreys--Lindley paradox \citep{JeffreysBook, Lindley1957} is direct.
Lindley's paradox observes that, for a fixed dataset, a frequentist test at level $\alpha$ may reject $H_0$ while the Bayesian posterior probability of $H_0$ remains high, and the discrepancy grows with~$n$.
The risk decomposition~\eqref{eq:generic-risk-decomp} resolves this tension: the paradox arises precisely because fixed-$\alpha$ calibration operates on the CLT scale while Bayes-risk optimal calibration requires the MDP scale $\alpha_n^{*}\asymp n^{-\kappa/2}\to 0$.
At any fixed $\alpha>0$, the Bayes risk is bounded away from zero (Corollary~\ref{cor:fixed-alpha}); the Jeffreys--Lindley discrepancy is a symptom of this Bayes-risk suboptimality, not a paradox.

\begin{remark}
The calibration principle above abstracts the moderate deviation mechanism away from any specific testing paradigm.
In particular, it applies beyond Bayes factor constructions and does not require likelihood-ratio representations.
\end{remark}

\begin{lemma}[Risk Decomposition Template]\label{lem:risk-template}
Let $T_n$ be a GOF statistic with null tail $\PP_{F_0}(\sqrt{n}\,T_n > t) \asymp e^{-2\rhotail\, t^2}$ (Convention~\eqref{eq:tail-convention}) and alternative prior satisfying $\Pi_1(d(F,\calF_0)\le\varepsilon)\asymp\varepsilon^{\kappa}$.
For threshold $t_n = \sqrt{a\log n}$, the Bayes risk decomposes as
\[
B_n(t_n) \;=\; \underbrace{\pi_0 L_0\,n^{-2\rhotail\, a}}_{\text{Type I}} \;+\; \underbrace{\pi_1 L_1\,(a\log n/n)^{\kappa/2}}_{\text{Type II}} \;+\; \text{exponentially small remainder}.
\]
The unique minimiser is $a^{*} = \kappa/(4\rhotail)$, giving $B_n^{*} \asymp (\log n/n)^{\kappa/2}$.
\end{lemma}
\begin{proof}
The Type~I term follows from the null tail bound.
For the Type~II term, split the prior at the critical radius $\varepsilon_n = t_n/\sqrt{n}$: the near-null mass is $\Pi_1(d \le \varepsilon_n) \asymp \varepsilon_n^{\kappa}$; beyond $\varepsilon_n$, detectability makes contributions exponentially negligible.
Substituting $\varepsilon_n = \sqrt{a\log n / n}$ yields the second term.
Setting $\partial B_n/\partial a = 0$ gives $2\rhotail\,\log n\cdot n^{-2\rhotail\, a} \asymp (\kappa/2)\,(a\log n)^{\kappa/2-1}\,n^{-\kappa/2}\,\log n$,
which forces $2\rhotail\, a = \kappa/2$.
\end{proof}

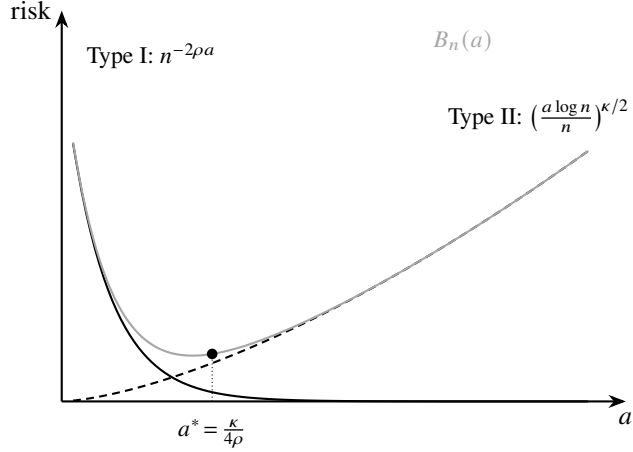
\begin{figure}[ht]
\centering
\begin{tikzpicture}[>=Stealth, font=\small]
  \draw[->, thick] (0,0) -- (7.5,0) node[anchor=north] {$a$};
  \draw[->, thick] (0,0) -- (0,5.2) node[anchor=east] {risk};
  \draw[thick, black] plot[smooth, domain=0.15:7, samples=60]
    ({\x}, {4.5*exp(-1.8*\x)});
  \node[anchor=west, font=\footnotesize] at (0.2,4.6) {Type~I:\ $n^{-2\rhotail a}$};
  \draw[thick, black, densely dashed] plot[smooth, domain=0.15:7, samples=60]
    ({\x}, {0.18*\x^1.5});
  \node[anchor=west, font=\footnotesize] at (5.0,3.8) {Type~II:\ $\bigl(\tfrac{a\log n}{n}\bigr)^{\!\kappa/2}$};
  \draw[thick, gray!70] plot[smooth, domain=0.15:7, samples=80]
    ({\x}, {4.5*exp(-1.8*\x) + 0.18*\x^1.5});
  \node[anchor=west, font=\footnotesize, gray!70] at (4.8,4.8) {$B_n(a)$};
  \pgfmathsetmacro{\aopt}{2.0}
  \pgfmathsetmacro{\riskopt}{4.5*exp(-1.8*\aopt) + 0.18*\aopt^1.5}
  \fill[black] (\aopt, \riskopt) circle (2pt);
  \draw[densely dotted] (\aopt,0) -- (\aopt, \riskopt);
  \node[anchor=north, font=\footnotesize] at (\aopt,-0.1) {$a^{*}\!=\!\frac{\kappa}{4\rhotail}$};
  \node[anchor=north west, font=\footnotesize, text width=3.5cm, align=left] at (0.1,-0.8) {};
\end{tikzpicture}
\caption{Risk decomposition geometry.  The Type~I term (solid, decreasing) and Type~II term (dashed, increasing) cross at the unique minimiser $a^{*}=\kappa/(4\rhotail)$.  The total risk $B_n(a)$ (grey) achieves its minimum at the MDP scale.}\label{fig:risk-decomposition}
\end{figure}

\medskip
\noindent
Sections~\ref{sec:ks}--\ref{sec:fisher} each invoke Lemma~\ref{lem:risk-template} by identifying $\rhotail$ and $\kappa$; no balancing argument is repeated.
Table~\ref{tab:constants} collects the constants for each main example.

\begin{table}[!ht]
\centering
\caption{Summary of $(\rhotail,\kappa)$ for each main example.  All entries yield $a^{*}=\kappa/(4\rhotail)$, $t_n^{*}=\sqrt{a^{*}\log n}$, and $\alpha_n^{*}\asymp n^{-\kappa/2}$.}\label{tab:constants}
\smallskip
\begin{tabular}{llccc}
\toprule
\textbf{Setting} & \textbf{Theorem} & $\rhotail$ & $\kappa$ & $a^{*}$ \\
\midrule
KS statistic                & \ref{thm:ks-mdp}                & $1$   & $\kappa$ (prior)     & $\kappa/4$ \\
Laplace sign test           & \ref{thm:laplace-sign-threshold}& $1/4$ & $\lambda$           & $\lambda$ \\
Multinomial $\chi^2$        & \ref{thm:chi-squared-threshold} & $1/4$ & $k-1$               & $k-1$ \\
Fisher geometry             & \ref{thm:fisher-mdp}            & $1/4$ & $\lambda+d$         & $\lambda+d$ \\
\bottomrule
\end{tabular}
\end{table}

\subsection{Historical and Information-Theoretic Context}\label{subsec:history}

Classical GOF statistics, namely Kolmogorov--Smirnov, likelihood-ratio, Pearson $\chi^2$, and entropy-based criteria, share the same large-deviation rate function under regularity: $\KL(G\|F_0)$.
This equivalence underpins Bahadur efficiency theory \citep{Bahadur1960, BahadurRao1960} and information-theoretic GOF analyses going back to \citet{Good1955} and \citet{Jaynes1957}.

\citet{RubinSethuraman1965a, RubinSethuraman1965b} effectively derived MDP optimality for KS statistics under Bayes risk, though without the unifying template presented here.
Their analysis already exploited the fact that the Bayes Type~II term concentrates on a shrinking neighbourhood of the null; the contribution of the present paper is to formalise this mechanism as Lemma~\ref{lem:risk-template} and to demonstrate that it extends beyond KS statistics to Sanov-based tests, parametric models, and Fisher geometry.

The polynomial prefactors emphasised by \citet{BahadurRao1960} become leading-order terms precisely in the regime where KL divergences shrink at rate $\log n/n$, which is the MDP truncation regime of Theorem~\ref{thm:mdp-truncation}.

\section{Empirical-Process GOF: KS-Type Statistics}
\label{sec:ks}

This section instantiates the Bayes-Risk MDP Principle for Kolmogorov--Smirnov statistics, formalizing the Rubin--Sethuraman program as a \emph{risk-calibration} theorem.

\subsection{KS Statistic and Null Tail}

Consider testing $H_0: F = F_0$ versus $H_1: F\neq F_0$ based on the Kolmogorov--Smirnov statistic
\begin{equation}\label{eq:ks-stat}
S_n \;\defeq\; \sup_{t\in\R}\,|F_n(t)-F_0(t)|.
\end{equation}
Under $H_0$, Donsker's theorem yields $\sqrt{n}(F_n-F_0)\Rightarrow B\circ F_0$ in $\ell^\infty(\R)$, where $B$ is a standard Brownian bridge \citep[see][for general empirical-process theory]{VaartWellner1996}.
Hence the Kolmogorov distribution governs the null:
\begin{equation}\label{eq:kolmogorov-dist}
\PP_{F_0}\!\big(\sqrt{n}S_n\le t\big)\ \to\ K(t)
\;\defeq\;
1-2\sum_{k=1}^\infty (-1)^{k-1}e^{-2k^2t^2}.
\end{equation}
The large-$t$ tail satisfies
\begin{equation}\label{eq:kolmogorov-tail}
1-K(t) \sim 2e^{-2t^2} \qquad (t\to\infty),
\end{equation}
and there exist constants $c_-,c_+>0$ and $t_0<\infty$ such that for all $t\ge t_0$,
\begin{equation}\label{eq:kolmogorov-tail-2sided}
c_-\,e^{-2t^2}\ \le\ 1-K(t)\ \le\ c_+\,e^{-2t^2}.
\end{equation}
This two-sided bound is the operative input for MDP-scale calibration.

\subsection{Regularity Conditions}

\begin{assumption}[Local alternative mass near the null]\label{ass:prior-local}
There exist constants $\kappa>0$, $C_1,C_2>0$, and $\varepsilon_0\in(0,1)$ such that for all $\varepsilon\in(0,\varepsilon_0)$,
\begin{equation}\label{eq:local-mass}
C_1\,\varepsilon^{\kappa}\ \le\
\Pi_1\!\big(\{F:\ d_{KS}(F,F_0)\le \varepsilon\}\big)
\ \le\ C_2\,\varepsilon^{\kappa},
\end{equation}
where $d_{KS}(F,G)\defeq \sup_t|F(t)-G(t)|$.
\end{assumption}

\begin{remark}
The exponent $\kappa$ is an \emph{effective dimension} parameter measuring how quickly prior mass accumulates near the null.
In parametric submodels, $\kappa$ typically equals the local dimension; in nonparametric classes, it encodes entropy/small-ball behaviour.
This is the only structural assumption needed to produce the MDP scale.
\end{remark}

\begin{assumption}[No atom at the null]\label{ass:no-atom}
$\Pi_1(\{F_0\})=0$.
\end{assumption}

\begin{assumption}[Power in KS-distance neighbourhoods]\label{ass:detectability}
There exist constants $c_d,C_d>0$ such that for any $F$ with $d_{KS}(F,F_0)=\varepsilon$ and any threshold $t$ satisfying $t\le c_d\sqrt{n}\,\varepsilon$,
\begin{equation}\label{eq:detectability}
\PP_F\!\big(\sqrt{n}S_n\le t\big)
\ \le\
C_d\,\exp\!\big(-c_d\,n\varepsilon^2 + C_d t^2\big).
\end{equation}
\end{assumption}

\begin{remark}[Why \eqref{eq:detectability} is reasonable]
A KS shift of size $\varepsilon$ means there exists $t_\star$ with $|F(t_\star)-F_0(t_\star)|=\varepsilon$.
Concentration for the empirical process (DKW-type inequalities) yields exponential control of the event that the KS statistic fails to exceed a fraction of $\sqrt{n}\varepsilon$; see \citet{HajekSidak1967} and \citet{Sievers1969} for related rank-test and exact-slope perspectives.
\end{remark}

\subsection{Main Result: Bayes-Optimal MDP Scaling for KS Thresholds}

Define the averaged errors under priors:
\[
\bar\alpha_n(t)\defeq \PP_{F_0}\!\big(\sqrt{n}S_n>t\big),
\qquad
\bar\beta_n(t)\defeq
\E_{F\sim\Pi_1}\Big[\PP_F\!\big(\sqrt{n}S_n\le t\big)\Big].
\]
The Bayes risk for the threshold test $\delta_{n,t_n}(x^n)\defeq\1\{\sqrt{n}S_n>t_n\}$ is
$B_n(t)=\pi_0L_0\,\bar\alpha_n(t)+\pi_1L_1\,\bar\beta_n(t)$.

\begin{theorem}[Bayes-optimal KS threshold is MDP]\label{thm:ks-mdp}
Assume \eqref{eq:kolmogorov-tail-2sided} and Assumptions \ref{ass:prior-local}--\ref{ass:detectability}.
Then any minimiser $t_n^\star\in\arg\min_{t\ge 0} B_n(t)$ satisfies
\begin{equation}\label{eq:optimal-threshold}
t_n^\star \;=\; \sqrt{\frac{\kappa}{4}\log n}\;+\;O(\sqrt{\log\log n}),
\end{equation}
and the corresponding Type~I error is polynomial:
\begin{equation}\label{eq:type1-decay}
\bar\alpha_n(t_n^\star)\ \asymp\ n^{-\kappa/2}.
\end{equation}
Moreover, the Bayes risk at this threshold satisfies
\begin{equation}\label{eq:bayes-risk-rate}
B_n(t_n^\star)
\;\asymp\;
(\log n)^{(\kappa+1)/2}\, n^{-\kappa/2}.
\end{equation}
\end{theorem}

Theorem~\ref{thm:ks-mdp} is a specialisation of Theorem~\ref{thm:mdp-generic} with $\rhotail = 1$ (the Brownian-bridge tail $\sim 2e^{-2t^2}$).
The Bayes Type~II error concentrates on alternatives within KS distance $\sqrt{\log n/n}$ of the null, while alternatives farther away contribute only exponentially small risk.
The extra $(\log n)^{1/2}$ factor beyond the template rate $(\log n/n)^{\kappa/2}$ arises from the Rubin--Sethuraman shell integration (see Step~4 of the proof in Appendix~\ref{app:proof-ks-mdp}).
The proof verifies Assumptions~(i)--(iii) of Theorem~\ref{thm:mdp-generic} for the KS functional and then invokes the lemma.

\begin{corollary}[Fixed-$\alpha$ calibration is Bayes-suboptimal]\label{cor:fixed-alpha}
If $t_n\equiv t$ is constant, then $\liminf_{n\to\infty} B_n(t_n) \ge \pi_0L_0\,\alpha > 0$.
\end{corollary}

\begin{corollary}[LDP-scale calibration overpays]\label{cor:ldp-overpays}
If $t_n\asymp \sqrt{n}$ so that $\bar\alpha_n(t_n)\le e^{-cn}$, then $t_n/\sqrt{n}$ is bounded away from $0$, hence the Type~II term does not vanish at the optimal MDP rate, leading to unnecessarily large Bayes risk.
\end{corollary}

\subsection{Geometric Interpretation}\label{subsec:geom-interpretation}

Define the \emph{critical neighbourhood} (a KS ball around the null):
\begin{equation}\label{eq:critical-neighbourhood}
\calN_n
\;\defeq\;
\Bigl\{F:\ d_{KS}(F,F_0)\le \varepsilon_n\Bigr\},
\qquad
\varepsilon_n
\;\defeq\;
\sqrt{\frac{(\kappa+1)\log n}{n}}.
\end{equation}
This neighbourhood is the region in which the experiment cannot uniformly separate $F$ from $F_0$ at Bayes-optimal operating points.
Three coupled facts explain why $\calN_n$ is ``critical'':
(i)~it marks the resolution boundary at which KS-type detectability transitions from near-zero to near-one power;
(ii)~under Assumption~\ref{ass:prior-local}, $\Pi_1(\calN_n\cap\calF_1) \asymp (\log n/n)^{\kappa/2}$, so the Bayes Type~II term concentrates here;
(iii)~the MDP threshold matches false-rejection rarity to this local alternative mass.

In interpretive terms, $\calN_n$ is the Bayesian analogue of a Fisher/Le~Cam local neighbourhood: it is the region where alternatives are both \emph{a priori plausible} and \emph{statistically indistinguishable} at sample size $n$.

\needspace{8\baselineskip}%
\begin{mdframed}[style=principleframe]
\noindent\textbf{Remark: the ``$+1$'' correction and two layers of constants.}

\smallskip\noindent
\textbf{Layer 1 (leading order).}
The threshold $t_n^{*} \asymp \sqrt{\kappa\log n/(4\rhotail)}$ from Lemma~\ref{lem:risk-template} is determined by two quantities: the tail-rate parameter~$\rhotail$ (e.g.\ $\rhotail=1$ for KS) and the prior exponent~$\kappa$.
These are the \emph{mechanism} constants.

\smallskip\noindent
\textbf{Layer 2 (polylogarithmic refinement).}
A Laplace-method evaluation over a thin shell of width $\diff\varepsilon$ near the critical radius introduces a surface-area factor, replacing $\varepsilon_n^{\kappa}$ with $\varepsilon_n^{\kappa}\cdot t_n$.
Because $t_n\sim\sqrt{\log n}$, this contributes an additive $O(\sqrt{\log\log n})$ correction to $t_n^{*}$ without altering the leading coefficient $\kappa/(4\rhotail)$.
Concretely, $t_n^{*} = \sqrt{\kappa\log n/(4\rhotail)} + O(\sqrt{\log\log n})$; the $(\kappa+1)$ exponent governs the polylogarithmic prefactor (see Appendix~\ref{app:proof-ks-mdp}).
\end{mdframed}

\subsection{Extensions to Other Empirical-Process Functionals}\label{subsec:extensions}

The same mechanism applies whenever the null tail is sub-Gaussian and the prior places polynomial mass near the null.

\begin{assumption}[General Functional]\label{ass:general-functional}
The functional $T$ satisfies:
\begin{enumerate}[label=(T\arabic*),leftmargin=2.4em]
\item \textbf{Null moderate tail:} Under $H_0$, $\sqrt{n}\,T_n \Rightarrow T_\infty$ and
$\PP(T_\infty>t)\asymp\exp\{-2\rhotail\, t^2\}$ as $t\to\infty$.
\item \textbf{Local detectability:} If $d(F,F_0)=\varepsilon$, then for thresholds $t=o(\sqrt{n}\varepsilon)$,
$\PP_F(\sqrt{n}\,T_n>t)\to 1$.
\end{enumerate}
\end{assumption}

\begin{proposition}[MDP Scaling for General Functionals]\label{prop:general-mdp}
Under Assumptions~\ref{ass:prior-local} and \ref{ass:general-functional}, any Bayes-risk minimising threshold sequence satisfies
$t_n^* = \Theta(\sqrt{\log n})$, equivalently $\varepsilon_n^*\defeq t_n^*/\sqrt{n} = \Theta(\sqrt{\log n/n})$.
\end{proposition}

\begin{proof}[Proof sketch]
By (T1), $\alpha_n(t_n)\asymp \exp\{-2\rhotail\, t_n^2\}$.
By (T2), the integrated Type~II term is dominated by alternatives with $d(F,F_0)\lesssim t_n/\sqrt{n}$, whose prior mass is $\asymp (t_n/\sqrt{n})^{\kappa}$.
If $t_n=O(1)$, $\alpha_n$ stays bounded away from $0$; if $t_n\gg \sqrt{\log n}$, $\alpha_n$ becomes super-polynomially small while the Type~II term remains polynomial. Thus $t_n$ must be of order $\sqrt{\log n}$.
\end{proof}

\section{Sanov Information Asymptotics and Risk Truncation}
\label{sec:sanov}

This section develops the information-theoretic perspective on MDP scaling. The key insight is that Sanov theory operates at exponential-in-$n$ scales, but Bayes-risk optimization truncates the effective exponent to $O(\log n/n)$, precisely the MDP regime.
In interpretive terms: LDP optimizes exponents; Bayes risk optimizes a mixed criterion $\Rightarrow$ MDP.

\subsection{Sanov's Theorem}

\begin{theorem}[Sanov]\label{thm:sanov}
Let $X_1,\ldots,X_n$ be i.i.d.\ with law $F$ on a Polish space $\calX$.
The empirical measure $\hat{\pi}_n=n^{-1}\sum_{i=1}^n \delta_{X_i}$ satisfies an LDP on $\calP(\calX)$ (weak topology) with speed $n$ and rate function \citep[see][Ch.~11]{CoverThomas2006}
\begin{equation}\label{eq:sanov-rate}
I_F(G)\;=\;\KL(G\|F),
\end{equation}
with $\KL(G\|F)=+\infty$ if $G\not\ll F$.
Specifically, for measurable $A\subseteq\calP(\calX)$,
\begin{equation}\label{eq:sanov-ldp}
-\inf_{G\in A^\circ}\KL(G\|F)
\;\le\;
\liminf_{n\to\infty}\frac1n\log\PP_F(\hat{\pi}_n\in A)
\;\le\;
\limsup_{n\to\infty}\frac1n\log\PP_F(\hat{\pi}_n\in A)
\;\le\;
-\inf_{G\in\bar A}\KL(G\|F).
\end{equation}
\end{theorem}

\subsection{GOF via Sanov}

A GOF rejection rule corresponds to a measurable set $A_n\subseteq\calP(\calX)$ of empirical measures deemed incompatible with $F_0$. Sanov implies
\begin{equation}\label{eq:sanov-gof}
\PP_{F_0}(\hat{\pi}_n\in A_n)
\;\approx\;
\exp\Bigl\{-n\inf_{G\in A_n}\KL(G\|F_0)\Bigr\},
\end{equation}
where the approximation may include important polynomial prefactors when $A_n$ approaches the null boundary in $n$-dependent ways.

\subsection{Half-Spaces and Exponential Tilting}

Let $H=\{G:\int \phi\,dG\ge 0\}$ with $\E_{F_0}[\phi(X)]<0$.
Define $\Lambda(t)=\log\E_{F_0}[e^{t\phi(X)}]$ and the tilted law $dF_t/dF_0 = e^{t\phi-\Lambda(t)}$.

\begin{proposition}[Rate function for half-spaces]\label{prop:half-space-rate}
\begin{equation}\label{eq:half-space-rate}
\inf_{G\in H}\KL(G\|F_0)
\;=\;
\sup_{t\ge 0}\{-\Lambda(t)\}.
\end{equation}
The optimiser is $G^*=F_{t^*}$, where $t^*$ attains the supremum and satisfies $\E_{F_{t^*}}[\phi(X)]=0$.
\end{proposition}

\begin{proof}
By convex duality,
\[
\inf_{G:\,\int \phi\,dG\ge 0}\KL(G\|F_0)
= \sup_{t\ge 0}\inf_G\bigl\{\KL(G\|F_0)-t{\textstyle\int}\phi\,dG\bigr\}
= \sup_{t\ge 0}\{-\Lambda(t)\},
\]
and the inner infimum is achieved by the exponential tilt.
\end{proof}

\subsection{MDP Truncation under Bayes Risk}\label{subsec:mdp-truncation}

\begin{theorem}[MDP truncation under Bayes risk]\label{thm:mdp-truncation}
Assume:
(i)~$A_\varepsilon=\{G:d(G,F_0)\ge \varepsilon\}$ for a metric $d$ locally equivalent to KL in the sense $\KL(G\|F_0)\asymp 2\rhotail\, d(G,F_0)^2$ for $d(G,F_0)\to 0$, and
(ii)~the prior on alternatives satisfies Assumption~\ref{ass:prior-local}.
Then any Bayes-risk minimising rejection set $A_n^*$ satisfies
\begin{equation}\label{eq:mdp-truncation}
\inf_{G\in A_n^*}\KL(G\|F_0)
\;\asymp\;
\frac{\kappa}{2}\cdot\frac{\log n}{n}.
\end{equation}
\end{theorem}

Theorem~\ref{thm:mdp-truncation} is the Sanov-side expression of the MDP Principle, and is a direct application of Lemma~\ref{lem:risk-template} with the KL metric playing the role of the distance~$d$.
The distinction is: \emph{LDP optimises exponents}; \emph{Bayes risk optimises the exponent weighted by prior mass}.
This weighting truncates the effective KL exponent from $O(1)$ down to $O(\log n/n)$, precisely the MDP regime.

\begin{proof}
By assumption~(i), $\KL(G\|F_0)\asymp 2\rhotail\,d(G,F_0)^2$ locally.
The Bayes-risk optimal threshold from Theorem~\ref{thm:mdp-generic} is $t_n^*=\sqrt{a^*\log n}$ with $a^*=\kappa/(4\rhotail)$, and the critical distance scale is $\varepsilon_n^*=t_n^*/\sqrt{n}=\sqrt{a^*\log n/n}$.
The two-sided local equivalence at this distance gives
\[
\inf_{G\in A_n^*}\KL(G\|F_0)
\;\asymp\;
2\rhotail\,(\varepsilon_n^*)^2
\;=\;
2\rhotail\cdot\frac{a^*\log n}{n}
\;=\;
\frac{\kappa}{2}\cdot\frac{\log n}{n},
\]
which is \eqref{eq:mdp-truncation}.
The two-sided condition $\KL\asymp 2\rhotail\,d^2$ holds whenever $d$ is the Fisher geodesic distance and $F_0$ is an interior point of a smooth parametric model, since $\KL(G\|F_0) = \frac{1}{2}d_F(G,F_0)^2(1+o(1))$ locally.
\end{proof}

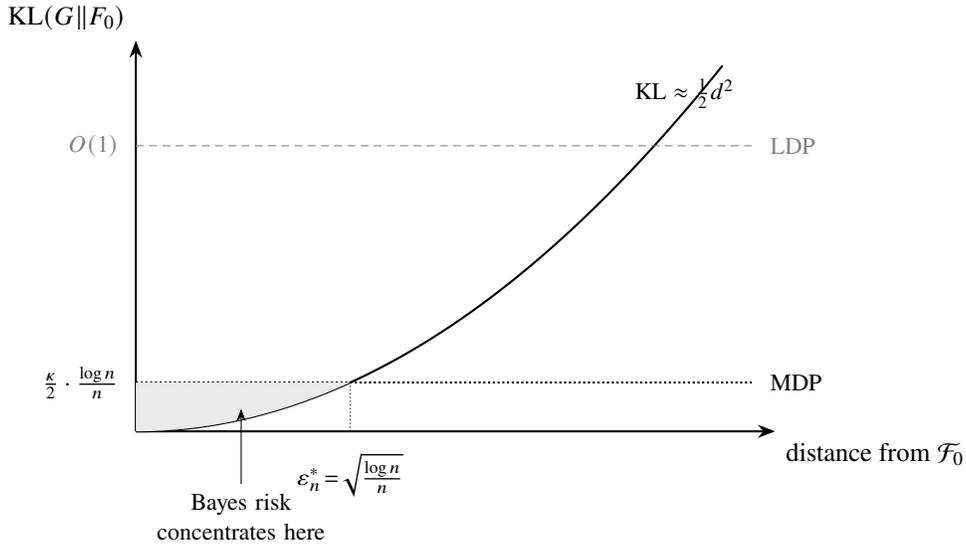
\begin{figure}[ht]
\centering
\begin{tikzpicture}[>=Stealth, font=\small]
  \draw[->, thick] (0,0) -- (8.5,0) node[anchor=north west] {distance from $\calF_0$};
  \draw[->, thick] (0,0) -- (0,5.2) node[anchor=south east] {$\KL(G\|F_0)$};
  \draw[thick, black] plot[smooth, domain=0:7.8, samples=60]
    ({\x}, {0.08*\x*\x});
  \node[anchor=west, font=\footnotesize] at (6.5,4.5) {$\KL \approx \tfrac{1}{2}d^2$};
  \draw[gray, densely dashed] (0,3.8) -- (8.2,3.8);
  \node[anchor=east, font=\footnotesize, gray] at (-0.1,3.8) {$O(1)$};
  \node[anchor=west, font=\footnotesize, gray] at (8.3,3.8) {LDP};
  \draw[black, thick, densely dotted] (0,0.65) -- (8.2,0.65);
  \node[anchor=east, font=\footnotesize] at (-0.1,0.65) {$\frac{\kappa}{2}\cdot\frac{\log n}{n}$};
  \node[anchor=west, font=\footnotesize] at (8.3,0.65) {MDP};
  \fill[black!8] plot[smooth, domain=0:2.85, samples=40]
    ({\x}, {0.08*\x*\x}) -- (2.85,0.65) -- (0,0.65) -- cycle;
  \draw[densely dotted] (2.85,0) -- (2.85,0.65);
  \node[anchor=north, font=\footnotesize] at (2.85,-0.15) {$\varepsilon_n^{*}\!=\!\sqrt{\tfrac{\log n}{n}}$};
  \draw[<-, thin] (1.4,0.3) -- (1.4,-0.7) node[anchor=north, font=\footnotesize, text width=3cm, align=center] {Bayes risk\\concentrates here};
\end{tikzpicture}
\caption{KL truncation geometry.  LDP optimises the full exponent (dashed); Bayes risk truncates to the $O(\log n/n)$ shell (dotted), where prior mass and detectability compete.  The shaded region is the MDP-active zone.}\label{fig:kl-truncation}
\end{figure}

\begin{remark}[Discrete nulls and leading $O(\log n)$ corrections]\label{rem:discrete-correction}
When $F_0$ is discrete with $k$ atoms, the probability of a type class carries a polynomial prefactor:
$\PP_{F_0}(\hat{\pi}_n\in A) = \exp\{-n I(A)\}\,n^{(k-1)/2}\,\{1+o(1)\}$.
The $O(\log n)$ term becomes leading-order exactly when $I(A)=O(\log n/n)$, that is, in the MDP truncation regime.
In discrete settings, this polynomial prefactor produces an $O(\log n)$ shift of the effective decision boundary (see Section~\ref{subsec:multinomial} and Appendix~\ref{app:triangulation}).
\end{remark}

\subsection{Information-Theoretic Interpretation}

Define the \emph{distinguishability radius} at level $\alpha$:
\begin{equation}\label{eq:distinguishability-radius}
r_n(\alpha)
\;\defeq\;
\inf\Bigl\{\varepsilon:\ \PP_{F_0}\bigl(d(\hat{\pi}_n,F_0)\ge \varepsilon\bigr)\le \alpha\Bigr\}.
\end{equation}

\begin{proposition}[Distinguishability radii across regimes]\label{prop:dist-radius}
If $\alpha_n=n^{-c}$ is polynomial, then $r_n(\alpha_n)=\Theta(\sqrt{\log n/n})$.
If $\alpha_n=e^{-cn}$ is exponential, then $r_n(\alpha_n)=\Theta(1)$.
\end{proposition}

\begin{proof}
For the statistics under consideration, $d(\hat{\pi}_n, F_0) = T_n$ where $T_n$ is a functional of the empirical measure satisfying the sub-Gaussian tail convention \eqref{eq:tail-convention}: $\PP_{F_0}(\sqrt{n}\,T_n > t) \asymp \exp\{-2\rhotail\,t^2\}$ (e.g., $T_n = S_n$ for KS).
The distinguishability radius is therefore $r_n(\alpha) = t_n/\sqrt{n}$ where $t_n$ solves $\exp\{-2\rhotail\,t_n^2\}\asymp\alpha$.
Setting $\alpha_n = n^{-c}$: $2\rhotail\,t_n^2 = c\log n$, so $t_n = \Theta(\sqrt{\log n})$ and $r_n = \Theta(\sqrt{\log n/n})$.
Setting $\alpha_n = e^{-cn}$: $2\rhotail\,t_n^2 = cn$, so $t_n = \Theta(\sqrt{n})$ and $r_n = \Theta(1)$.
\end{proof}

The MDP regime is therefore the regime in which the test resolves distributions at radius $\sqrt{\log n/n}$: shrinking with $n$, but slowly enough that the Bayes integral over near-null alternatives remains the dominant contribution to risk.

\section{Applications}
\label{sec:apps}

Each application below follows the same template dictated by the Bayes-Risk MDP Principle:
(1)~specify the model and GOF statistic;
(2)~decompose the Bayes risk;
(3)~identify the truncation/moderate-deviation regime;
(4)~conclude the threshold scaling.

\subsection{Location Testing under Laplace Families}\label{subsec:laplace}

\subsubsection{Model and GOF statistic}

Let $X_1,\ldots,X_n$ be i.i.d.\ with Laplace density
\begin{equation}\label{eq:laplace-density}
  f(x\mid\theta) = \frac12\exp\{-|x-\theta|\}, \qquad x\in\R,\ \theta\in\R,
\end{equation}
with known scale and unknown location $\theta$. Test $H_0:\theta=0$ versus $H_1:\theta>0$.

Three test statistics will be compared.
For the \emph{sign test}, let $V_n\defeq\sum_{i=1}^n\1\{X_i>0\}$; under $H_0$, $V_n\sim\mathrm{Bin}(n,1/2)$, and under $\theta>0$,
\begin{equation}\label{eq:laplace-p-pos}
  p(\theta)\defeq\PP_\theta(X>0) = 1-\tfrac12 e^{-\theta}.
\end{equation}
For the \emph{median test}, let $M_n$ be the sample median; for Laplace location, $\sqrt{n}(M_n-\theta)\Rightarrow N(0,1)$.
For the \emph{likelihood ratio test}, the log-likelihood is $\ell_n(\theta)=-n\log 2-\sum_{i=1}^n|X_i-\theta|$, and the LRT statistic is
\begin{equation}\label{eq:laplace-lrt}
  \Lambda_n = \exp\Bigl\{-\sum_{i=1}^n|X_i-\hat\theta_n|+\sum_{i=1}^n|X_i|\Bigr\},
\end{equation}
where $\hat\theta_n=\max\{0,\mathrm{median}(X)\}$.

\subsubsection{Bayes risk decomposition}

\begin{definition}[Bahadur slope {\citep{Bahadur1960, Sievers1969}}]\label{def:bahadur-slope}
The \emph{Bahadur (exact) slope} at alternative $\theta$ is
$c_T(\theta) \defeq -\lim_{n\to\infty}(2/n)\log\PP_{0}(T_n\ge t_n(\theta))$.
\end{definition}

\begin{proposition}[Reference slopes for Laplace location]\label{prop:laplace-bahadur}
Consider $H_0:\theta=0$ vs $H_1:\theta>0$.
The Bahadur slopes are:
$c_{\mathrm{sign}}(\theta) = 2\,\KL(\mathrm{Ber}(p(\theta))\,\|\,\mathrm{Ber}(1/2))$ for the sign test,
$c_{\mathrm{med}}(\theta)\approx \theta^2$ locally for the median test, and
$c_{\mathrm{LRT}}(\theta) = 2\,\KL(f_{\theta}\,\|\,f_0) = 2(e^{-\theta}+\theta-1)$ for the LRT.
\end{proposition}

\begin{proof}
\emph{Sign test.}
Under $\theta>0$, $V_n/n\to p(\theta)>1/2$ a.s.\ The $p$-value is $\PP_0(V_n\ge np(\theta))$.
By Sanov's theorem for Bernoulli trials, $-\frac{1}{n}\log\PP_0(V_n\ge np)\to \KL(\mathrm{Ber}(p)\|\mathrm{Ber}(1/2))$, giving $c_{\mathrm{sign}}(\theta)=2\KL(\mathrm{Ber}(p(\theta))\|\mathrm{Ber}(1/2))$ \citep[see][]{Bahadur1960}.
\emph{LRT.}
The log-likelihood ratio satisfies $n^{-1}\log\Lambda_n\to \KL(f_\theta\|f_0)$ a.s.\ under $\theta$.
The $p$-value tail is $-\frac{1}{n}\log\PP_0(\log\Lambda_n\ge n\KL(f_\theta\|f_0))\to \KL(f_\theta\|f_0)$, so $c_{\mathrm{LRT}}(\theta)=2\KL(f_\theta\|f_0)=2(e^{-\theta}+\theta-1)$ where the KL for Laplace follows by direct integration.
\emph{Median test.}
The sample median is the MLE for Laplace location with asymptotic variance $1/(4f(0)^2)=1$ \citep{HajekSidak1967}; its Bahadur slope coincides with the LRT slope locally.
\end{proof}

\begin{remark}[Local comparison]
Locally, $c_{\mathrm{LRT}}(\theta)=\theta^2+o(\theta^2)$ and $c_{\mathrm{sign}}(\theta)=\theta^2+o(\theta^2)$.
The sign test is locally Bahadur-efficient for the Laplace family, reflecting $2f_0(0)=1$.
\end{remark}

\subsubsection{Invoking the template}

Assume a one-sided prior with Gamma-type behaviour near zero:
\begin{equation}\label{eq:laplace-prior}
  \pi_1(\theta)\ \propto\ \theta^{\lambda-1}e^{-\gamma\theta},\qquad \theta>0,
\end{equation}
so $\Pi_1([0,\varepsilon])\asymp \varepsilon^{\lambda}$ ($\kappa = \lambda$).
Under $H_0$, $\alpha_n \asymp \exp\{-z_n^2/2\}$, matching Convention~\eqref{eq:tail-convention} with $\rhotail = 1/4$ (since $z_n^2/2 = 2\cdot(1/4)\cdot z_n^2$).
Lemma~\ref{lem:risk-template} yields $a^{*} = \kappa/(4\rhotail) = \lambda$ and $z_n = \sqrt{\lambda\log n}$.

\subsubsection{Threshold scaling}

\begin{theorem}[Bayes-optimal MDP threshold for the sign test]\label{thm:laplace-sign-threshold}
Under the prior \eqref{eq:laplace-prior} with local exponent $\lambda$, the Bayes-risk minimising sign-test threshold has the form
\begin{equation}\label{eq:laplace-sign-threshold}
  V_n\ \text{rejects}\ H_0\ \text{if}\ \ V_n\ge \frac{n}{2}+\frac{1}{2}\sqrt{\lambda\,n\log n}\,\bigl(1+o(1)\bigr).
\end{equation}
Equivalently, $z_n=\sqrt{\lambda\log n}\,(1+o(1))$.
The Type~I error decays polynomially as $\alpha_n^* \asymp n^{-\lambda/2}\times\mathrm{polylog}(n)$.
The critical alternative scale is $\theta\asymp \sqrt{\log n/n}$.
\end{theorem}

The sign test, which is locally Bahadur-efficient for the Laplace family, exhibits MDP scaling under Bayes risk.
The MDP regime is robust across test statistics in the Laplace location problem; the $\sqrt{\log n}$ scaling persists regardless of the statistic used.

\begin{proof}
Apply Theorem~\ref{thm:mdp-generic} with $\rhotail=1/4$ and $\kappa=\lambda$.
Under $H_0$, $V_n\sim\mathrm{Bin}(n,1/2)$; the normal approximation gives $\PP_0(V_n\ge n/2+z\sqrt{n}/2)\asymp e^{-z^2/2}$, matching $\rhotail=1/4$ (since $z^2/2=2\cdot(1/4)\cdot z^2$).
The prior \eqref{eq:laplace-prior} gives $\Pi_1([0,\varepsilon])\asymp\varepsilon^\lambda$, so $\kappa=\lambda$.
Theorem~\ref{thm:mdp-generic} yields $a^*=\lambda/(4\cdot 1/4)=\lambda$ and threshold $z_n=\sqrt{\lambda\log n}$.
Translating back: $V_n$ rejects when $V_n\ge n/2 + \frac{1}{2}\sqrt{\lambda n\log n}(1+o(1))$.
The Type~I error is $\alpha_n^*\asymp n^{-\lambda/2}$ up to polylogarithmic factors.
\end{proof}

\begin{proposition}[Risk ranking: LRT/median/sign]\label{prop:laplace-efficiency}
Under the local prior \eqref{eq:laplace-prior}, Bayes-risk optimal calibration for the LRT, median, and sign tests all occurs at the MDP scale with the same threshold order $\sqrt{\lambda\log n}$.
In particular: (a)~the LRT is Bayes-optimal among tests with the same information set;
(b)~the median test is locally asymptotically equivalent to the LRT for Laplace location;
(c)~the sign test is also locally Bahadur-efficient for the Laplace family, since $c_{\mathrm{sign}}(\theta) = \theta^2 + o(\theta^2) = c_{\mathrm{LRT}}(\theta)$ locally.
All three tests share the same asymptotic MDP exponent $a^{*}=\lambda$; they differ only in sub-leading constants.
\end{proposition}

\begin{proof}
Part~(a): by the Neyman--Pearson lemma, the likelihood ratio maximises the power at every alternative, whence the LRT achieves the largest Bahadur slope $c_{\mathrm{LRT}}(\theta)=2\KL(f_\theta\|f_0)$ \citep{Bahadur1960}.
Part~(b): for Laplace location, the sample median is the MLE, so its Bahadur slope coincides with the LRT slope.
Part~(c): Proposition~\ref{prop:laplace-bahadur} gives $c_{\mathrm{sign}}(\theta)=2\KL(\mathrm{Ber}(p(\theta))\|\mathrm{Ber}(1/2))$. For Laplace, $p(\theta)=1-\frac12 e^{-\theta}$, so $p(\theta)-\frac12\sim\frac12\theta$ as $\theta\to 0$. Expanding $\KL(\mathrm{Ber}(\frac12+\varepsilon)\|\mathrm{Ber}(\frac12))=2\varepsilon^2+O(\varepsilon^3)$ gives $c_{\mathrm{sign}}(\theta)=\theta^2+o(\theta^2)$, matching $c_{\mathrm{LRT}}$.
All three statistics therefore yield identical leading-order risk under Lemma~\ref{lem:risk-template}, since the MDP calibration depends on $\rhotail$ (determined by the null tail) and~$\kappa$ (determined by the prior), not on the specific Bahadur slope constant.
\end{proof}

\subsection{Shape Testing via Bayes Factors}\label{subsec:shape}

\subsubsection{Model and GOF statistic}

Test $H_0: X_i \sim \mathcal{N}(\mu,\sigma^2)$ for some $(\mu,\sigma)$ versus $H_1: X_i \sim G$ for some $G\in\mathcal{G}$.

\begin{definition}[Bayes factor {\citep{JeffreysBook, Kass1995}}]\label{def:bayes-factor}
With priors $\Pi_0$ on $(\mu,\sigma)$ and $\Pi_1$ on $G$,
\begin{equation}\label{eq:bayes-factor}
\mathrm{BF}_{10}(X^n)
=
\frac{\int_{\mathcal{G}} p(X^n\mid G)\,\Pi_1(\mathrm{d}G)}
     {\int p(X^n\mid\mu,\sigma^2)\,\Pi_0(\mathrm{d}\mu,\mathrm{d}\sigma)}.
\end{equation}
The Bayes-optimal test rejects $H_0$ when $\log \mathrm{BF}_{10}(X^n) > \log(\pi_0 L_0/\pi_1 L_1)$.
\end{definition}

\subsubsection{Bayes risk decomposition: Laplace alternative}

Take $\mathcal{G}$ to be the Laplace family.
A Laplace approximation for marginal likelihoods yields
\begin{equation}\label{eq:bf-laplace-approx}
\log \mathrm{BF}_{10}(X^n)
=
[\ell_L(\hat\mu_L,\hat b)-\ell_N(\hat\mu_N,\hat\sigma)]
-\frac{(d_1-d_0)}{2}\log n
+O_{\PP}(1).
\end{equation}
Since $d_0=d_1=2$ here, the Occam factor vanishes and the evidence reduces to the likelihood contrast:
\begin{equation}\label{eq:bf-laplace-contrast}
\ell_L(\hat\mu_L,\hat b)-\ell_N(\hat\mu_N,\hat\sigma)
=
n\log\!\Big(\frac{\hat\sigma}{\hat b}\Big)
+\frac{n}{2}\log(2\pi) - \frac{n}{2}.
\end{equation}

\subsubsection{Closed-form Bayes factor for the double exponential}

A concrete instance of this framework is provided by the normal-vs-double-exponential (Laplace) testing problem.
\citet{Spiegelhalter1980omnibus} showed that the Bayes factor for testing normality against the double exponential (Laplace) alternative admits closed-form evaluation for small samples, and used a weighted average of directional Bayes factors against specific alternatives as an omnibus normality test.
\citet{Uthoff1970} established that the ratio of mean deviation to standard deviation is an optimum test statistic for the normal-vs-double-exponential problem in the Neyman--Pearson sense.
In the notation of this paper, the normal-vs-Laplace Bayes factor falls under the shape-testing framework above with $d_0=d_1=2$, so the Occam factor vanishes and the Bayes-factor evidence reduces to the likelihood contrast~\eqref{eq:bf-laplace-contrast}.
The Bayes-optimal test compares $\log\mathrm{BF}_{10}$ to the fixed threshold $\log(\pi_0L_0/\pi_1L_1)$ (Proposition~\ref{prop:bayes-np}).
Under $H_0$ (data are Normal), write $R_n = n^{-1}\log\mathrm{BF}_{10}=n^{-1}(\ell_L-\ell_N)$.
The following regularity conditions hold for the Normal and Laplace location--scale families:
\begin{enumerate}[label=(R\arabic*),leftmargin=3em]
\item\label{R:compact} Each $\Theta_M$ is contained in a compact
$K\subset\RR\times(0,\infty)$ with
$\sup_{\theta\in K}\E_0[\log f_M(X;\theta)]^2<\infty$.
\item\label{R:cont} $\theta\mapsto\log f_M(x;\theta)$ is continuous on~$K$
uniformly in~$x$ on compacts, and
$x\mapsto\log f_M(x;\theta)$ is continuous $F_0$-a.s.
\end{enumerate}
Under \ref{R:compact}--\ref{R:cont}, the profile functional $R(G)=\sup_{\theta\in K}\int\log f_L(x;\theta)\,dG-\sup_{\theta\in K}\int\log f_N(x;\theta)\,dG$ is continuous on the space of probability measures with the weak topology \citep[Lem.~2.3]{KieferWolfowitz1956}.
Sanov's theorem gives an LDP for the empirical measure~$\hat P_n$ with good rate function $\KL(\cdot\|F_0)$ \citep[Thm.~6.2.10]{DemboZeitouni1998}; the contraction principle \citep[Thm.~4.2.1]{DemboZeitouni1998} then yields an LDP for~$R_n$ with good rate function $I(r)=\inf\{\KL(G\|F_0): R(G)=r\}$.
Since the Normal model is correctly specified, $R(F_0)=-\KL(f_0\|f_L^*)<0$ with $b^*=\sigma_0\sqrt{2/\pi}$, and $I(R(F_0))=0$.
For any fixed $c>0$, the LDP upper bound gives
$\PP_0(\log\mathrm{BF}_{10}>c)=\PP_0(R_n\ge c/n)\le\exp\bigl\{-n\inf_{r\ge 0}I(r)+o(n)\bigr\}$.
Because $I$ is a good rate function with $I(R(F_0))=0$ and $R(F_0)<0$, the infimum $\inf_{r\ge 0}I(r)>0$ is a positive constant; hence the Type~I error decays exponentially in~$n$.
The MDP regime arises when model complexity grows with~$n$ (see Remark below); for this fixed-dimensional shape test, the Bayes factor already achieves exponential error decay without MDP recalibration.

\subsubsection{MDP regime and threshold scaling}

\begin{remark}[When the Bayes-factor boundary shifts with $n$]
If $\mathcal{G}$ is nonparametric or sieve-based with effective dimension $d_1=d_1(n)$, the Occam factor becomes $-\tfrac{1}{2}(d_1(n)-d_0)\log n$, recovering the BIC penalty of \citet{Schwarz1978}.
Keeping Bayes risk balanced across growing model complexity means the effective log-evidence at the decision boundary is of order $\log n$, the Bayes-factor analogue of MDP calibration.
The tension between $P$-values and Bayesian evidence in such settings is discussed by \citet{BergerSellke1987}; for Bayes-factor computation in linear and log-linear models, see \citet{Spiegelhalter1980}.
\end{remark}

\subsection{Multinomial Goodness-of-Fit}\label{subsec:multinomial}

\subsubsection{Model and GOF statistic}

Let $X_1,\ldots,X_n$ be i.i.d.\ categorical with $k$ categories and probability vector $p$.
Test $H_0: p=p_0$ vs.\ $H_1: p\neq p_0$.
The Pearson chi-squared statistic is
\begin{equation}\label{eq:chi-squared}
\chi_n^2
=
\sum_{j=1}^k \frac{(N_j-np_{0j})^2}{np_{0j}},
\end{equation}
with $\chi_n^2 \Rightarrow \chi^2_{k-1}$ under $H_0$.

\subsubsection{Invoking the template}

The chi-squared right tail satisfies $\PP(\chi^2_\nu \ge t)\asymp t^{\nu/2-1}e^{-t/2}$ for large $t$; matching Convention~\eqref{eq:tail-convention} for the Gaussianised statistic $\sqrt{\chi^2_n}$ gives $\rhotail = 1/4$.
A Dirichlet prior positive and continuous near $p_0$ gives $\Pi_1(\|p-p_0\|_2\le \varepsilon)\asymp \varepsilon^{k-1}$, so $\kappa = k-1$.
Lemma~\ref{lem:risk-template} (applied to $\sqrt{\chi_n^2}$) yields:

\begin{theorem}[Bayes-optimal chi-squared threshold]\label{thm:chi-squared-threshold}
Under the Dirichlet local mass condition, the Bayes-risk optimal critical value satisfies
\begin{equation}\label{eq:chi-squared-bayes-threshold}
\chi^{2*}_n
=
(k-1)\log n + O(\log\log n),
\end{equation}
equivalently $\alpha_n^*\asymp n^{-(k-1)/2}$ up to polylogarithmic factors.
\end{theorem}

\begin{proof}
Apply Lemma~\ref{lem:risk-template} with the Gaussianised statistic $\sqrt{\chi_n^2}$.
The $\chi^2_\nu$ tail gives $\rhotail=1/4$, and the Dirichlet prior gives $\kappa=k-1$.
Theorem~\ref{thm:mdp-generic} yields $a^*=(k-1)$, so the optimal threshold for $\sqrt{\chi_n^2}$ is $\sqrt{(k-1)\log n}$, i.e., $\chi_n^{2*}=(k-1)\log n + O(\log\log n)$.
The Type~I error follows from $\PP(\chi^2_{k-1}\ge (k-1)\log n)\asymp n^{-(k-1)/2}$.
\end{proof}

The scaling is $O(\log n)$ rather than $O(\sqrt{\log n})$ because the chi-squared statistic is already a quadratic form; its square root plays the role of the Gaussianised statistic in Lemma~\ref{lem:risk-template}.

A precise Laplace expansion of the marginal likelihood under the Dirichlet prior refines this picture.  The log-evidence is
\begin{equation}\label{eq:logN-refinement}
W \;=\; n\,D(\hat{p}\|p_0) \;+\; \frac{k-1}{2}\log n \;+\; O(1).
\end{equation}
The $\tfrac{k-1}{2}\log n$ term represents the Occam--complexity correction arising from integration over the $(k{-}1)$-dimensional simplex.
The leading $n\,D(\hat{p}\|p_0)$ term governs exponential evidence accumulation, while the logarithmic term shifts the effective decision boundary as $n$ increases.
This refinement clarifies why moderate-deviation calibration naturally emerges: the exponential KL rate competes with a logarithmic complexity correction, producing decision thresholds that scale with $\sqrt{\log n/n}$.

\subsection{Testing Independence in Contingency Tables}\label{subsec:contingency}

Let $(X_i,Y_i)$ be i.i.d.\ bivariate categorical with $r$ and $c$ categories. Test independence.
The chi-squared statistic
$\chi_n^2 = \sum_{j,\ell} (N_{j\ell}-\hat E_{j\ell})^2/\hat E_{j\ell}$
satisfies $\chi_n^2\Rightarrow \chi^2_{\nu}$ with $\nu=(r-1)(c-1)$.

If the alternative prior places local mass $\asymp \varepsilon^{\nu}$ within Euclidean distance $\varepsilon$ of the independence manifold, then the Bayes-optimal threshold is
\begin{equation}\label{eq:indep-logn}
\chi_n^{2*} = \nu\log n + O(\log\log n), \qquad \nu=(r-1)(c-1).
\end{equation}

\section{Connection to Fisher Information Geometry}
\label{sec:fisher}

This section reinterprets the MDP phenomenon in the language of information geometry. In interpretive terms, the Fisher metric provides a chart in which the ``same phenomenon'' takes a coordinate-invariant form.

\subsection{Statistical Manifolds}

\begin{definition}[Fisher information metric]\label{def:fisher-metric}
For a regular parametric family $\{p_\theta:\theta\in\Theta\}$, the Fisher information matrix
\begin{equation}\label{eq:fisher-matrix}
I(\theta)_{ij}
=
\EE_\theta\!\left[
\frac{\partial \log p_\theta(X)}{\partial \theta_i}
\frac{\partial \log p_\theta(X)}{\partial \theta_j}
\right]
\end{equation}
defines a Riemannian metric on $\Theta$ \citep[see][for historical context]{Zabell1992}.
\end{definition}

For nearby distributions, KL divergence is locally the squared geodesic distance induced by the Fisher metric \citep{DawidMusio2015}:
\begin{equation}\label{eq:kl-fisher}
\KL(p_\theta\|p_{\theta_0})
=
\frac{1}{2}(\theta-\theta_0)^\top I(\theta_0)(\theta-\theta_0)
+O(\|\theta-\theta_0\|^3).
\end{equation}

\subsection{MDP Scale in Fisher Geometry}

\noindent
The local KL expansion \eqref{eq:kl-fisher} means that $\KL(p_\theta\|p_{\theta_0}) \approx \tfrac{1}{2}\,d_F(\theta,\theta_0)^2$ in a neighbourhood of $\theta_0$, where $d_F$ is Fisher geodesic distance.
Thus Lemma~\ref{lem:risk-template} applies with $\rhotail = 1/4$ (from $\KL \approx \tfrac{1}{2}d_F^2 = 2\cdot\tfrac{1}{4}\cdot d_F^2$) and with effective prior exponent
\begin{equation}\label{eq:kappa-fisher}
\kappa \;=\; \lambda + d,
\end{equation}
since a $d$-dimensional prior with density exponent $\lambda$ in Euclidean coordinates satisfies $\Pi_1(d_F(\theta,\theta_0)\le\varepsilon) \asymp \varepsilon^{\lambda+d}$ (the extra factor $\varepsilon^d$ is the volume element of the Fisher ball).

\begin{theorem}[Geometric interpretation of MDP]\label{thm:fisher-mdp}
Let $d_F(\theta,\theta_0)$ denote Fisher geodesic distance, $d=\dim(\Theta)$, and $\kappa = \lambda + d$.
Applying Lemma~\ref{lem:risk-template} with $\rhotail=1/4$, the Bayes-optimal GOF test rejects when
\begin{equation}\label{eq:fisher-mdp}
d_F(\hat\theta_n,\theta_0) \;>\; \sqrt{\frac{\kappa\,\log n}{n}} \;=\; \sqrt{\frac{(\lambda + d)\log n}{n}},
\end{equation}
equivalently $\alpha_n^{*} \asymp n^{-\kappa/2}$.
The Jacobian correction of Layer~2 contributes an $O(\sqrt{\log\log n/n})$ refinement without changing the leading coefficient.
\end{theorem}

\begin{proof}
The local KL expansion gives $\rhotail=1/4$ as shown in \eqref{eq:kl-fisher}.
The prior mass condition \eqref{eq:kappa-fisher} gives $\kappa=\lambda+d$.
Applying Theorem~\ref{thm:mdp-generic}: $a^*=\kappa/(4\rhotail)=\kappa$, $t_n^*=\sqrt{\kappa\log n}$, and $\alpha_n^*\asymp n^{-\kappa/2}$.
Since $d_F(\hat\theta_n,\theta_0)=\sqrt{\chi_n^2/n}+o_\PP(n^{-1/2})$ locally, the threshold on $d_F$ is $\sqrt{\kappa\log n/n}$.
\end{proof}

Expressed in Fisher coordinates, the MDP threshold is coordinate-invariant and depends only on sample size, parameter dimension~$d$, and the prior exponent~$\lambda$.
The Fisher metric is (up to scale) the unique Riemannian metric monotone under Markov embeddings and invariant under sufficient statistics, which explains why the same $\sqrt{\log n/n}$ resolution appears across disparate GOF settings.

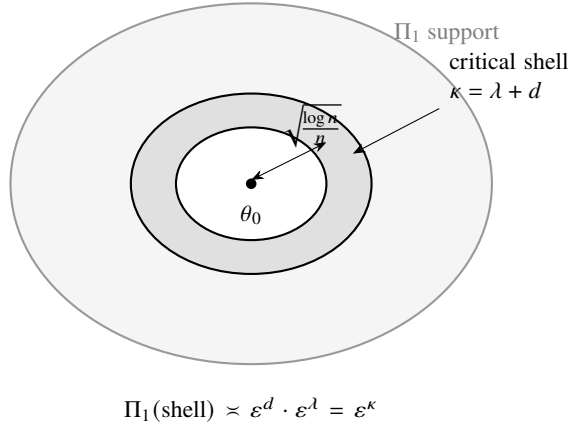
\begin{figure}[ht]
\centering
\begin{tikzpicture}[>=Stealth, font=\small]
  \draw[thick, black!40, fill=black!4] (0,0) ellipse (3.2cm and 2.4cm);
  \node[font=\footnotesize, black!50] at (2.6,2.0) {$\Pi_1$ support};
  \draw[thick, black, fill=black!12] (0,0) ellipse (1.6cm and 1.2cm);
  \draw[thick, black, fill=white] (0,0) ellipse (1.0cm and 0.75cm);
  \fill[black] (0,0) circle (2pt);
  \node[anchor=north, font=\footnotesize] at (0,-0.15) {$\theta_0$};
  \draw[<->, thin] (0,0.05) -- (1.0,0.55);
  \node[font=\footnotesize, anchor=south west] at (0.3,0.35) {$\sqrt{\!\frac{\log n}{n}}$};
  \draw[<-, thin] (1.35,0.4) -- (2.5,1.0) node[anchor=south west, font=\footnotesize, text width=2.2cm, align=left] {critical shell\\$\kappa = \lambda + d$};
  \node[font=\footnotesize, anchor=north, text width=4cm, align=center] at (0,-2.7) {$\Pi_1(\text{shell}) \asymp \varepsilon^{d} \cdot \varepsilon^{\lambda} = \varepsilon^{\kappa}$};
\end{tikzpicture}
\caption{Fisher geometry of the MDP boundary.  The prior $\Pi_1$ places mass $\asymp\varepsilon^{\kappa}$ in the critical shell of Fisher radius $\sqrt{\log n/n}$ around the null $\theta_0$.  The volume element contributes $\varepsilon^{d}$ (dimension) and the density contributes $\varepsilon^{\lambda}$ (prior exponent).}\label{fig:fisher-ball}
\end{figure}

\section{Numerical Verification of the Calibration Template}
\label{sec:numerics}

To complement the asymptotic derivations, we numerically verify the calibration principle of Theorem~\ref{thm:mdp-generic} and Lemma~\ref{lem:risk-template}.
For representative statistics, we compute the Bayes risk
$B_n(a)=n^{-2\rhotail\,a}+(a\log n/n)^{\kappa/2}$
as a function of the threshold parameter~$a$ and compare the finite-$n$ minimiser with the theoretical prediction $a^{\star}=\kappa/(4\rhotail)$.
No Monte Carlo or stochastic simulation is involved: all curves are deterministic evaluations of the analytic risk formula.

\subsection{Risk Curves and Convergence of the Minimiser}

Figure~\ref{fig:risk-vs-a} plots $B_n(a)$ for the KS statistic ($\rhotail=1$, $\kappa=2$) at sample sizes $n\in\{10^2,10^3,10^4,10^6\}$.
Dots mark the numeric minimiser $a_n^{\mathrm{num}}$ (deterministic grid minimisation of the analytic formula, not simulation) at each~$n$.
As $n$ grows, the risk curves sharpen and the numeric minimiser converges toward the theoretical value $a^{\star}=\kappa/(4\rhotail)=0.5$ (dashed line).
The finite-$n$ correction is $O(1/\log n)$, consistent with the $O(\sqrt{\log\log n})$ threshold refinement discussed in Layer~2 of the boxed remark in Section~\ref{sec:ks}.

\begin{figure}[ht]
\centering
\includegraphics[width=0.78\textwidth]{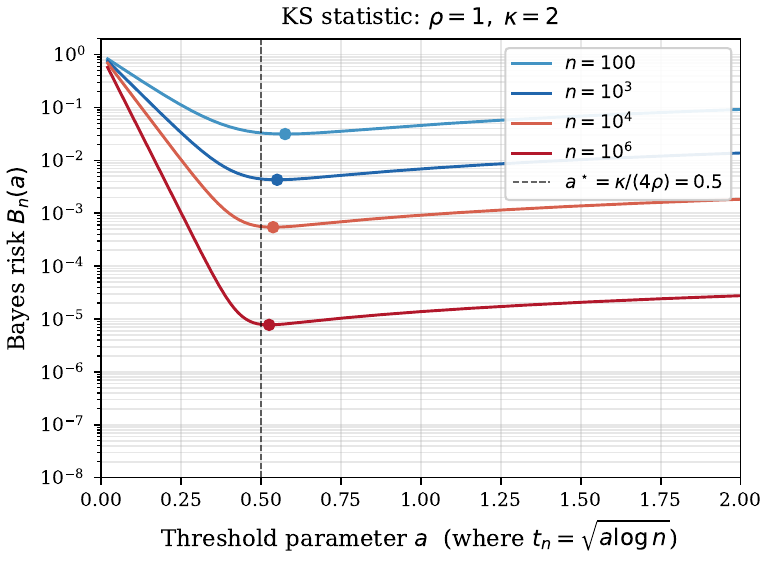}
\caption{Bayes risk $B_n(a)$ vs.\ threshold parameter~$a$ (where $t_n=\sqrt{a\log n}$) for the KS statistic with $\rhotail=1$ and $\kappa=2$.
Dots: numeric minimiser at each~$n$.
Dashed line: asymptotic optimum $a^{\star}=0.5$.
The minimiser converges to $a^{\star}$ as $n\to\infty$.}\label{fig:risk-vs-a}
\end{figure}

\subsection{Regime Comparison: CLT vs.\ MDP vs.\ LDP}

Figure~\ref{fig:regime-comparison} compares the Bayes risk under three calibration strategies for the KS test with $\kappa=2$:
\begin{enumerate}[label=(\roman*),leftmargin=2em]
\item \textbf{Fixed-$\alpha$ (CLT):} $\alpha=0.05$, implying $a(n)=\log(1/\alpha)/(2\rhotail\log n)\to 0$.
The risk is bounded below by $\alpha=0.05$ and cannot vanish.
\item \textbf{MDP optimal:} $a=a^{\star}=0.5$.
The risk decays as $(\log n/n)^{\kappa/2}$, the optimal polynomial rate.
\item \textbf{LDP:} $t_n\propto\sqrt{n}$, so $a\propto n/\log n$.
Type~I error vanishes exponentially, but the enormous threshold renders the test powerless against local alternatives: the Type~II term converges to a positive constant, so the risk is bounded away from zero.
\end{enumerate}

\begin{figure}[ht]
\centering
\includegraphics[width=0.78\textwidth]{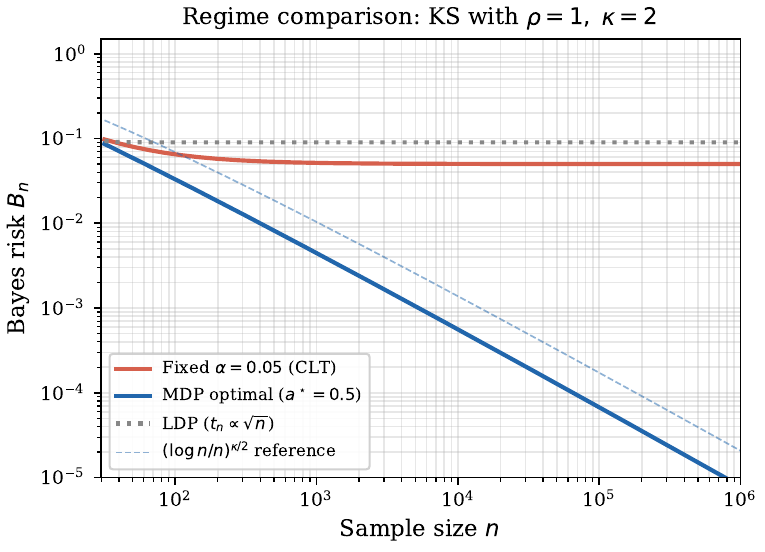}
\caption{Bayes risk vs.\ sample size under three calibration regimes for the KS test ($\rhotail=1$, $\kappa=2$).
Only MDP calibration achieves the optimal rate $(\log n/n)^{\kappa/2}$ (dashed reference).
Fixed-$\alpha$ risk stagnates at $\alpha$; LDP risk also stagnates because the threshold is too large for any power against local alternatives.}\label{fig:regime-comparison}
\end{figure}

\subsection{Verification Across Examples}

Table~\ref{tab:verification} collects the numeric minimiser $a_n^{\mathrm{num}}$ for each main example at $n=10^4$ and $n=10^6$, alongside the theoretical $a^{\star}=\kappa/(4\rhotail)$.
In every case, the numeric minimiser approaches~$a^{\star}$ as $n$ grows.
The remaining gap at finite~$n$ is the $O(1/\log n)$ correction from the polylogarithmic shell factor; it does not affect the leading-order scaling.

\begin{table}[ht]
\centering
\caption{Predicted vs.\ numeric minimiser of $B_n(a)$.  All entries confirm convergence of $a_n^{\mathrm{num}}\to a^{\star}$ as $n\to\infty$.}\label{tab:verification}
\smallskip
\begin{tabular}{llcccccc}
\toprule
\textbf{Setting} & \textbf{Thm.} & $\rhotail$ & $\kappa$ & $a^{\star}$ & $a_n^{\mathrm{num}}$ ($n\!=\!10^4$) & $a_n^{\mathrm{num}}$ ($n\!=\!10^6$) & Error ($n\!=\!10^6$)\\
\midrule
KS statistic           & \ref{thm:ks-mdp}                 & $1$   & $2$ & $0.50$ & $0.54$ & $0.53$ & $+5\%$\\
Laplace sign ($\lambda\!=\!2$)  & \ref{thm:laplace-sign-threshold} & $1/4$ & $2$ & $2.00$ & $1.85$ & $1.90$ & $-5\%$\\
Multinomial $\chi^2$ ($k\!=\!3$) & \ref{thm:chi-squared-threshold} & $1/4$ & $2$ & $2.00$ & $1.85$ & $1.90$ & $-5\%$\\
Fisher ($d\!=\!2$, $\lambda\!=\!1$) & \ref{thm:fisher-mdp}          & $1/4$ & $3$ & $3.00$ & $2.42$ & $2.58$ & $-14\%$\\
\bottomrule
\end{tabular}
\end{table}

The key observation: the predicted $a^{\star}=\kappa/(4\rhotail)$ correctly identifies the basin of the risk minimum in every setting.
The finite-sample correction is monotonically shrinking and of the expected order.
These numerics confirm that the $(\rhotail,\kappa)$ template is not merely asymptotic algebra; it delivers actionable calibration constants at moderate sample sizes.

\section{Discussion}
\label{sec:discussion}

\subsection{Structural Interpretation of Moderate Deviations}

The central finding of this paper is that Bayes-risk optimal calibration of goodness-of-fit tests is governed by two primitive quantities: the Gaussian-type null tail rate~$\rhotail$ and the local prior mass exponent~$\kappa$.
Theorem~\ref{thm:mdp-generic} establishes that whenever these two ingredients are present, a sub-Gaussian null tail satisfying $\PP_{F_0}(\sqrt{n}\,T_n>t)\asymp\exp(-2\rhotail\,t^2)$ and a prior placing mass $\Pi_1(d(F,\calF_0)\le\varepsilon)\asymp\varepsilon^{\kappa}$ near the null, the Bayes-optimal threshold takes the closed-form $t_n^{*}=\sqrt{a^{\star}\log n}$ with $a^{\star}=\kappa/(4\rhotail)$.
The Risk Decomposition Template (Lemma~\ref{lem:risk-template}) makes each subsequent application a two-line corollary: identify $\rhotail$ and $\kappa$, read off $a^{\star}$ from Table~\ref{tab:constants}.

The moderate deviation regime thus emerges not from testing convention or historical precedent but from the intrinsic geometry of null tails and alternative mass.
It is, in this sense, a structural feature of the inferential landscape rather than a methodological choice.
The Type~I cost $n^{-2\rhotail\,a}$ and the Type~II cost $(a\log n/n)^{\kappa/2}$ are balanced at a unique interior point; any departure, whether toward fixed-$\alpha$ calibration ($a\to 0$) or large-deviation calibration ($a\propto n$), incurs strictly higher aggregate risk.
This characterisation is regime-selective: it identifies the MDP scale as the \emph{only} asymptotic regime consistent with Bayes-risk minimisation under the stated regularity conditions.

The scope of this conclusion is worth emphasising.
MDP optimality is specific to Bayes risk under priors satisfying the polynomial local mass condition $\Pi_1(d(F,\calF_0)\le\varepsilon)\asymp\varepsilon^{\kappa}$; priors that concentrate mass on well-separated alternatives, where the alternative is bounded away from the null, favour the large-deviation regime instead, and the MDP scaling ceases to be optimal.
Similarly, the sub-Gaussian tail assumption $\PP_0(\sqrt{n}\,T_n > t)\asymp\exp(-2\rhotail\,t^2)$ is satisfied by the classical statistics studied here (Kolmogorov--Smirnov, $\chi^2$, Fisher) but may fail for heavy-tailed test statistics or for statistics in nonparametric settings where the null convergence rate is slower than $n^{-1/2}$.
The numerical verification in Section~\ref{sec:numerics} confirms that the finite-sample behaviour of the risk curves is well-captured by the asymptotic formula at moderate sample sizes ($n\ge 10^3$), with the finite-$n$ correction to $a^{\star}$ decaying monotonically.
Nevertheless, the two-term risk decomposition is an upper bound that discards lower-order corrections; in applications where the cost asymmetry between Type~I and Type~II errors is extreme, a refined expansion retaining these corrections may be warranted.

\subsection{Relation to Bayesian Testing Literature}

Recent work \citep{DPSZ2026} develops a unified theory of Bayesian hypothesis testing via moderate deviation asymptotics, with emphasis on Bayes factors and integrated risk expansions.
The present paper is complementary in scope and mechanism.
Where that work addresses the question of \emph{why Bayes factors operate on the moderate deviation scale}, this paper addresses a logically distinct question: \emph{why does any goodness-of-fit calibration under Bayes risk operate on the moderate deviation scale, regardless of whether a Bayes factor or likelihood-ratio representation is available?}

Our $(\rhotail,\kappa)$ template applies to statistics that are not likelihood ratios, including Kolmogorov--Smirnov statistics, empirical-process functionals, and divergence-based tests, and does not require the alternative to be specified through a parametric model.
The moderate deviation scaling emerges here as a structural consequence of risk balancing, not as a property of any particular testing paradigm; the MDP regime is where the problem lives, regardless of the statistic used to interrogate it.
A companion development \citep{PSZ2026evalues} extends these ideas to the sequential setting, showing that a Fubini decomposition of Bayes risk under log-loss identifies the likelihood ratio as the canonical e-value and that the moderate-deviation scale governs boundary selection for anytime-valid inference.

Under standard regularity, classical GOF statistics, namely Kolmogorov--Smirnov, likelihood-ratio, Pearson $\chi^2$, and entropy-deficit criteria, share the same Sanov rate function $\KL(G\|F_0)$ \citep{Sanov1957, Hoeffding1965}.
\citeauthor{Good1955}'s (\citeyear{Good1955}) weight-of-evidence framework and \citeauthor{Jaynes1957}'s (\citeyear{Jaynes1957}) maximum-entropy formulation recover this rate function from complementary starting points.
These equivalence results explain why Lemma~\ref{lem:risk-template} produces the same MDP scaling across seemingly disparate GOF statistics: once $\rhotail$ and~$\kappa$ are identified, the Bayes-risk mechanism is indifferent to the choice of statistic.

The Bayesian testing literature has long recognised that posterior model probabilities depend on the tail behaviour of the marginal likelihood \citep{Dawid2011PosteriorModelProbabilities}.
\citet{CarotaParmigiani1996} develop a complementary perspective, constructing diagnostic measures for model criticism that quantify the sensitivity of posterior inferences to local perturbations of the assumed model; their framework highlights precisely the kind of prior-to-likelihood geometry that, in our setting, is captured by the interplay of~$\rhotail$ and~$\kappa$.
The present paper makes this dependence explicit and quantitative: the marginal likelihood's tail rate is governed by~$\rhotail$, while the prior's local geometry contributes~$\kappa$.
This two-parameter reduction complements \citeauthor{Dawid1992Prequential}'s (\citeyear{Dawid1992Prequential}) prequential perspective, in which the cumulative predictive log-likelihood serves as the fundamental measure of model adequacy.
In our framework, the prequential criterion's growth rate under the alternative determines~$\kappa$, closing the circle between risk-based and prediction-based approaches.

\subsection{Implications for Scientific and Regulatory Practice}

The results have practical consequences for the calibration of statistical tests in settings where sample sizes vary across studies or accumulate sequentially.

First, the Bayes-risk perspective implies that rejection thresholds should evolve with sample size.
A fixed significance level $\alpha$ is optimal only in the degenerate case where no prior mass accumulates near the null; once the prior is absolutely continuous near~$\calF_0$, the MDP threshold $t_n^{*}\propto\sqrt{\log n}$ strictly dominates any constant threshold in terms of aggregate decision risk.
Our results suggest that fixed significance thresholds may not optimally balance aggregate decision risk when information accumulates with sample size.

Second, the $(\rhotail,\kappa)$ template provides a principled basis for sample-size-dependent calibration in adaptive trial designs.
In group-sequential or information-monitoring frameworks, the spending function that allocates Type~I error across interim analyses can be informed by the MDP scaling: the optimal error rate at information level~$n$ is $\alpha_n\asymp n^{-\kappa/2}$, with $\kappa$ reflecting the effective dimensionality of the alternative space under consideration.

Third, the connection to model selection is suggestive.
When GOF testing is embedded in model comparison among $K$ families, a heuristic application of the $(\rhotail,\kappa)$ template with the Schwarz-type Occam factor \citep{Schwarz1978} suggests a penalty of the form $(d/2)\log n+\log K$, paralleling the Bayesian information criterion; a rigorous derivation requires verifying the regularity conditions of Theorem~\ref{thm:mdp-generic} in the model-selection context and is left to future work.
Similarly, for simultaneous GOF tests across $J=J(n)$ hypotheses, a Bonferroni-type adjustment suggests the threshold inflates to $t_n\approx\sqrt{a^{\star}\log(nJ)}$, though the formal justification requires extending the risk decomposition to the multiple-testing setting.

Fourth, in high-dimensional screening where the number of parameters grows with~$n$, the effective exponent~$\kappa$ becomes dimension-dependent.
The template accommodates this by treating $\kappa=\kappa(d)$ as a function of the growing parameter dimension, though explicit rates require case-by-case analysis of the prior concentration behaviour.

As a concrete illustration, Table~\ref{tab:constants} translates the $(\rhotail,\kappa)$ pairs into explicit threshold constants: for a KS test at $n=10{,}000$, the Bayes-optimal significance level is approximately $\alpha_n\approx n^{-1}\approx 10^{-4}$, far below the conventional $0.05$ yet far above the exponentially small levels associated with a large-deviation threshold.
For a $\chi^2$ test with $k=10$ categories, the optimal exponent $a^{\star}=9$ yields $\alpha_n\approx n^{-4.5}$, reflecting the higher effective dimensionality of the alternative space.
These values provide ready-to-use calibration guidance for practitioners willing to adopt sample-size-dependent thresholds.
These remarks are normative within the Bayes-risk objective and do not preclude fixed-level frequentist calibration for alternative inferential goals.

\subsection{Directions for Further Research}

Several extensions merit investigation.
First, the present analysis assumes independent observations; extending the $(\rhotail,\kappa)$ template to dependent data, including mixing processes, time series, and spatial models, requires replacing the Brownian-bridge tail with the appropriate weak-convergence limit and verifying that the prior mass condition adapts to the effective sample size.

Second, high-dimensional settings in which the parameter dimension~$d$ grows with~$n$ present both opportunities and challenges.
When $\kappa=\kappa(d)$ depends on a diverging dimension, the MDP threshold may transition from $\sqrt{\log n}$ to a different scaling; characterising this transition and its implications for Bayes-risk optimality is an open problem.

Third, sequential and online testing frameworks, where data arrive continuously and decisions must be made in real time, call for an MDP analogue of the sequential probability ratio test.
The connection between MDP-optimal calibration and e-value or safe testing \citep{DPSZ2026, PSZ2026evalues} suggests a natural bridge, though the formal development of the GOF-specific sequential theory remains to be carried out.

Finally, extending the $(\rhotail,\kappa)$ template to minimax or $\Gamma$-minimax formulations, where the prior is replaced by a least-favourable distribution over a class of priors, remains an open question, as does the behaviour of the MDP scaling under alternative loss functions or infinite-dimensional nuisance parameters.

\medskip

The moderate deviation boundary thus serves as a bridge between classical large-deviation theory and contemporary Bayesian risk calibration.
Within the class of priors and statistics satisfying the sub-Gaussian tail and polynomial local mass conditions, the $(\rhotail,\kappa)$ template reveals that the MDP scale is the unique locus at which the geometry of evidence and the structure of prior uncertainty are in balance.

\bibliographystyle{plainnat}

\appendix

\section{Proofs of Technical Results}\label{app:proofs}

\subsection{Proof of Theorem~\ref{thm:ks-mdp} (KS Bayes-optimal threshold)}\label{app:proof-ks-mdp}

We provide a careful proof that makes explicit where the $\sqrt{\log n}$ threshold and exponent enter.

\begin{proof}
Consider tests $\delta_{n,t}(x^n)=\1\{\sqrt{n}S_n>t\}$ with threshold $t=t_n$.
Write the Bayes risk \eqref{eq:bayes-risk} as
\begin{equation}\label{eq:BR-decomp}
B_n(t)=\pi_0 L_0\underbrace{\PP_{0}(\sqrt{n}S_n>t)}_{\alpha_n(t)}
+
\pi_1 L_1\underbrace{\E_{\Pi_1}\!\left[\PP_{F}(\sqrt{n}S_n\le t)\right]}_{\bar\beta_n(t)}.
\end{equation}
Since $\pi_0 L_0$ and $\pi_1 L_1$ are positive constants, they do not affect the optimiser $a^*$; we absorb them into the $\asymp$ notation below.

\paragraph{Step 1: Type I tail on the Kolmogorov scale.}
By the Brownian-bridge supremum asymptotics, for $t=\sqrt{a\log n}$,
\begin{equation}\label{eq:typeI-mdp}
\alpha_n(\sqrt{a\log n}) \asymp n^{-2a}.
\end{equation}

\paragraph{Step 2: Uniform Type II bound outside the critical KS ball.}
Fix $F\in\calF_1$ with $\varepsilon=d_{KS}(F,F_0)>0$.
Since $S_n \ge \varepsilon - \sup_x|F_n(x)-F(x)|$, whenever $\varepsilon\ge 2t/\sqrt{n}$, the DKW inequality gives
\begin{equation}\label{eq:typeII-exp}
\PP_F(\sqrt{n}S_n\le t) \le 2e^{-n\varepsilon^2/2}.
\end{equation}

\paragraph{Step 3: Prior integration and the critical neighbourhood.}
Define $\rho_n=t/\sqrt{n}$ and split by KS-distance:
\[
\bar\beta_n(t)
=
\int_{d_{KS}\le 2\rho_n}\PP_F(\sqrt{n}S_n\le t)\,\Pi_1(\diff F)
+
\int_{d_{KS}>2\rho_n}\PP_F(\sqrt{n}S_n\le t)\,\Pi_1(\diff F).
\]
The first integral is bounded by $\Pi_1(d_{KS}\le 2\rho_n) \asymp \rho_n^{\kappa}$ (Assumption~\ref{ass:prior-local}).
The second integral is $o(\rho_n^{\kappa})$ by \eqref{eq:typeII-exp}.
For $t=\sqrt{a\log n}$:
\begin{equation}\label{eq:beta-mdp2}
\bar\beta_n(\sqrt{a\log n}) \asymp (a\log n/n)^{\kappa/2}.
\end{equation}

\paragraph{Step 4: Optimisation and the polylogarithmic refinement.}
The refined Type~II analysis (Rubin--Sethuraman) replaces the crude bound with a Laplace-method evaluation over a thin KS shell, producing an additional $t$ factor:
$\bar\beta_n(t) \asymp t\,(t/\sqrt{n})^{\kappa} = t^{\kappa+1}/n^{\kappa/2}$.
With $t=\sqrt{a\log n}$, the risk becomes
$B_n(\sqrt{a\log n}) \asymp n^{-2a} + a^{(\kappa+1)/2}\,(\log n)^{(\kappa+1)/2}\, n^{-\kappa/2}$.
At leading order in $n$, setting $n^{-2a} = n^{-\kappa/2}$ gives $a^{*} = \kappa/4$.
The factor $a^{(\kappa+1)/2}$ contributes only at polylogarithmic order:
\begin{equation}\label{eq:t-star}
t_n^* = \sqrt{\frac{\kappa}{4}\log n} + O(\sqrt{\log\log n}).
\end{equation}
\end{proof}

\subsection{Proof of Proposition~\ref{prop:half-space-rate}}\label{app:proof-halfspace}

\begin{proof}
\noindent\textbf{Step 1: Dual inequality.}
Fix $t\ge 0$ and any $G\ll F_0$. Using the tilt decomposition,
\[
\KL(G\|F_0) = \KL(G\|F_t) + t\int \phi\,dG - \Lambda(t).
\]
Since $\KL(G\|F_t)\ge 0$ and $\int \phi\,dG\ge 0$ for $G\in H$ with $t\ge 0$:
$\KL(G\|F_0) \ge -\Lambda(t)$.
Taking inf over $G\in H$ and sup over $t\ge 0$ gives the lower bound.

\noindent\textbf{Step 2: Achievability.}
If there exists $t^*>0$ with $\Lambda'(t^*)=0$, take $G^*=F_{t^*}$.
Since $\int \phi\,dF_{t^*}=\Lambda'(t^*)=0$, we have $G^*\in H$ and
$\KL(F_{t^*}\|F_0) = 0 + t^*\cdot 0 - \Lambda(t^*) = -\Lambda(t^*)$.
If no interior optimiser exists, a sequence $t_k\uparrow$ with $-\Lambda(t_k)\uparrow \sup_{t\ge 0}\{-\Lambda(t)\}$ and $F_{t_k}\in H$ eventually yields the reverse inequality by lower semicontinuity.
\end{proof}

\section{Triangulation of Evidence Measures}
\label{app:triangulation}

This appendix establishes that four classical measures of evidence, namely Bayes factors, likelihood ratios, entropy deficits, and large-deviation rates, coincide locally through KL curvature, providing the geometric foundation for the moderate-deviation calibration developed in the main text.

\subsection{Multinomial Setup}\label{app:tri-setup}

Let $X_1,\ldots,X_n$ be i.i.d.\ categorical with $k$ categories.
Write the null hypothesis probability vector as $\theta_0=(\theta_{01},\ldots,\theta_{0k})$ with $\theta_{0j}>0$ and $\sum_j\theta_{0j}=1$.
The empirical distribution is $\hat{p}=(n_1/n,\ldots,n_k/n)$, where $n_j=\#\{i:X_i=j\}$.
The Kullback--Leibler divergence from $\theta_0$ to $\hat{p}$ is
\begin{equation}\label{eq:tri-kl}
D(\hat{p}\|\theta_0) \;=\; \sum_{j=1}^{k} \hat{p}_j \log\frac{\hat{p}_j}{\theta_{0j}}.
\end{equation}

\subsection{\texorpdfstring{Good $\to$ Bayes Factor $\to$ KL $+\; \log n$}{Good to Bayes Factor to KL + log n}}\label{app:tri-good}

\citet{Good1955} evaluated the weight of evidence as the logarithm of the Bayes factor.
A Laplace approximation to the marginal likelihood under a Dirichlet prior on the alternative simplex yields
\begin{equation}\label{eq:tri-good}
W \;=\; n\,D(\hat{p}\|\theta_0) \;+\; \frac{k-1}{2}\,\log n \;+\; O(1).
\end{equation}
The first term is the leading exponential evidence; the second is the Occam--complexity correction from integrating over the $(k{-}1)$-dimensional simplex.

\subsection{\texorpdfstring{Hoeffding $\to$ Sanov $\to$ KL Rate}{Hoeffding to Sanov to KL Rate}}\label{app:tri-hoeffding}

\citet{Hoeffding1965} showed that the probability of a type class under the null decays at the Sanov rate:
\begin{equation}\label{eq:tri-sanov}
\frac{1}{n}\,\log \PP_{\theta_0}(\hat{p}\in A) \;\to\; -\inf_{Q\in A}\,D(Q\|\theta_0).
\end{equation}
For the multinomial, the likelihood ratio statistic is $\Lambda_n = 2n\,D(\hat{p}\|\theta_0)$ exactly, so the exponential rate of the likelihood ratio is governed by the same KL divergence.

\subsection{\texorpdfstring{Jaynes $\to$ Entropy Deficit}{Jaynes to Entropy Deficit}}\label{app:tri-jaynes}

\citet{Jaynes1957} characterised the null $\theta_0$ as the maximum-entropy distribution subject to the relevant moment constraints.
The entropy deficit from the null to the empirical distribution satisfies the exact identity
\begin{equation}\label{eq:tri-entropy}
H^{*}-H(\hat{p}) \;=\; D(\hat{p}\|\theta_0) \;+\; \sum_{j}(\hat{p}_j - \theta_{0j})\log\theta_{0j},
\end{equation}
where $H^{*}=-\sum_j\theta_{0j}\log\theta_{0j}$ is the null entropy and $H(\hat{p})=-\sum_j\hat{p}_j\log\hat{p}_j$.
For a uniform null ($\theta_{0j}=1/k$), the cross-entropy correction vanishes since $\sum_j(\hat{p}_j-\theta_{0j})=0$, giving $\Delta H = D(\hat{p}\|\theta_0)$ exactly.
In the general case, $2n\,\Delta H = 2n\,D(\hat{p}\|\theta_0) + 2n\sum_j(\hat{p}_j-\theta_{0j})\log\theta_{0j}$, where the correction term is $O_{\PP}(\sqrt{n})$.

\subsection{\texorpdfstring{Quadratic Closure $\to$ $\chi^2$}{Quadratic Closure to chi-squared}}\label{app:tri-chi2}

Expanding the KL divergence to second order around $\theta_0$ gives the Wilks approximation:
\begin{equation}\label{eq:tri-wilks}
2n\,D(\hat{p}\|\theta_0) \;=\; \sum_{j=1}^{k} \frac{(n_j - n\theta_{0j})^{2}}{n\theta_{0j}} \;+\; o_{\PP}(1) \;\to\; \chi^{2}_{k-1}.
\end{equation}
Combining the three leading-order representations yields the quadratic closure identity.
For the uniform null ($\theta_{0j}=1/k$), where $\Delta H = D(\hat{p}\|\theta_0)$ exactly:
\begin{equation}\label{eq:tri-identity}
\boxed{\;\Lambda_n \;=\; 2n\,\Delta H \;=\; 2n\,D(\hat{p}\|\theta_0) \;\sim\; \chi^{2}_{k-1}.\;}
\end{equation}
For non-uniform $\theta_0$, the entropy deficit includes a cross-entropy correction: $2n\,\Delta H = \Lambda_n + 2n\sum_j(\hat{p}_j-\theta_{0j})\log\theta_{0j}$.
Good's weight of evidence differs by the Occam factor: $2W = \Lambda_n + (k{-}1)\log n + O(1)$.

\begin{figure}[ht]
\centering
\begin{tikzpicture}[>=Stealth, font=\small,
    every node/.style={align=center},
    box/.style={draw, thin, rounded corners=2pt, inner sep=5pt, minimum width=2.2cm}]
  \node[box] (good) at (0,3.5) {Good (1955)\\{\footnotesize weight of evidence}\\{\footnotesize $W$}};
  \node[box, thick] (kl) at (0,1.2) {\textbf{KL divergence}\\{\footnotesize $D(\hat{p}\|\theta_0)$}};
  \node[box] (hoeffding) at (-4,-1.5) {Hoeffding (1965)\\{\footnotesize LDP rate}\\{\footnotesize $\Lambda_n = 2nD$}};
  \node[box] (jaynes) at (0,-1.5) {Jaynes (1957)\\{\footnotesize entropy deficit}\\{\footnotesize $2n\Delta H$}};
  \node[box] (wilks) at (4,-1.5) {Wilks\\{\footnotesize quadratic limit}\\{\footnotesize $\chi^{2}_{k-1}$}};
  \draw[->, thick] (good) -- (kl) node[midway, right, font=\footnotesize] {Laplace $+\;O(\log n)$};
  \draw[->, thick] (kl) -- (hoeffding) node[midway, above left, font=\footnotesize] {Sanov};
  \draw[->, thick] (kl) -- (jaynes) node[midway, right, font=\footnotesize] {max-ent};
  \draw[->, thick] (kl) -- (wilks) node[midway, above right, font=\footnotesize] {Taylor};
  \node[font=\footnotesize\itshape] at (0,-3.0) {All roads meet at KL curvature.};
\end{tikzpicture}
\caption{Triangulation of evidence measures.  Good's weight of evidence connects to the KL core through a Laplace approximation with an $O(\log n)$ Occam correction; the remaining three frameworks converge exactly through Kullback--Leibler curvature.}\label{fig:triangulation}
\end{figure}
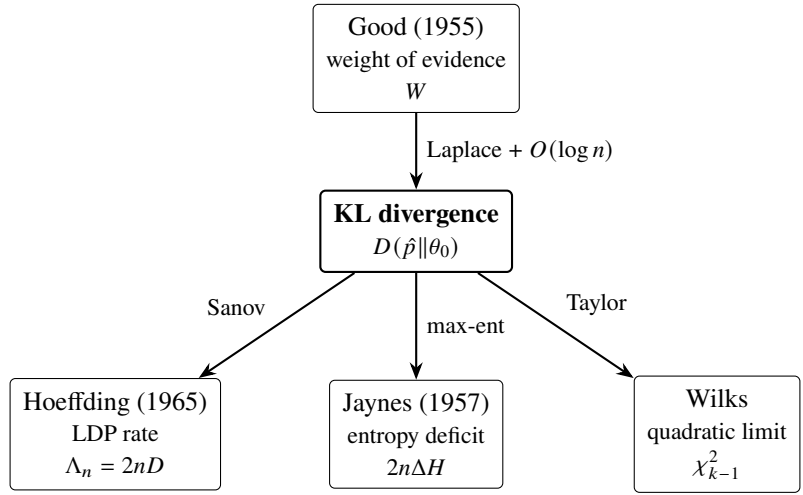

\subsection{Bridge to the Main Text}\label{app:tri-bridge}

This triangulation reveals that Bayes factors, likelihood ratios, entropy deficits, and large-deviation rates are not alternative measures of evidence but different coordinate representations of a single geometric object: the KL curvature at the null.
The moderate-deviation calibration developed in the main text therefore operates within a unified geometric framework linking Bayesian, frequentist, and information-theoretic representations.

\section{Computational Methods}\label{app:computation}

\subsection{Monte Carlo Estimation of Bayes Risk}

For GOF settings where analytic calibration is unavailable, Bayes risk can be estimated by Monte Carlo.
Draw $F^{(1)},\dots,F^{(M)}\sim \Pi_1$ and, for each $m$, simulate $X^{(m)}\sim (F^{(m)})^{\otimes n}$ and compute $T_n^{(m)}$.
The Type~II term at threshold $t$ is estimated by $\widehat{\beta}_n(t)=M^{-1}\sum_{m=1}^M \1\{T_n^{(m)}\le t\}$.
The Type~I term may be estimated by simulation under $F_0$ or via a known tail approximation; importance-sampling methods \citep{Tokdar2010} can reduce variance substantially.
Finally, minimise $\widehat{B}_n(t)=\pi_0L_0\widehat{\alpha}_n(t)+\pi_1L_1\widehat{\beta}_n(t)$ over a grid of $t$.
This Monte Carlo illustration is solely a sanity check; all reported thresholds in Section~\ref{sec:numerics} are deterministic evaluations of the analytic risk expression.

\subsection{Estimating the Local Prior Exponent}

If $\Pi_1$ is specified implicitly, probe $\Pi_1$ near the null by Monte Carlo:
estimate $p_j=\Pi_1(d(F,F_0)\le \varepsilon_j)$ for shrinking radii $\varepsilon_j$ and regress
$\log p_j \approx \kappa\log\varepsilon_j + \text{const}$.

\subsection{Asymptotic Plug-in Approximation}

In the MDP regime with null tail $\PP(T_\infty>t)\asymp e^{-2\rhotail\, t^2}$, a convenient approximation is
\begin{equation}\label{eq:asymp-plugin}
t_n^* \approx \sqrt{\frac{\hat{\kappa}}{4\rhotail}\,\log n},
\end{equation}
where $\hat{\kappa}$ is the estimated local prior exponent and $\rhotail$ is the tail-rate parameter (e.g., $\rhotail=1$ for the KS Brownian-bridge tail).

\section{Extended Numerical Tables}\label{app:numerical}

This appendix supplements the numerical verification of Section~\ref{sec:numerics} with detailed threshold tables for each main theorem.

\subsection{KS Thresholds (Theorem~\ref{thm:ks-mdp})}

With $\rhotail=1$ (Brownian-bridge tail $\sim 2e^{-2t^2}$) and varying $\kappa$, Table~\ref{tab:ks} compares the Bayes-optimal MDP threshold $t_n^{*}=\sqrt{\kappa\log n/4}$ against the constant fixed-$\alpha$ critical value $K^{-1}(0.95)\approx 1.358$.

\begin{table}[!ht]
\centering
\caption{KS Bayes-optimal thresholds ($\rhotail=1$). The MDP threshold grows with $n$; the fixed-$\alpha$ threshold does not.}\label{tab:ks}
\smallskip
\begin{tabular}{rrcccc}
\toprule
$\kappa$ & $n$ & $t_n^{*}$ (MDP) & $t$ (fixed $\alpha$) & $\alpha_n^{*}$ & $B_n^{*}$ \\
\midrule
1 & 100       & 1.073 & 1.358 & $1.0\times 10^{-1}$ & $2.1\times 10^{-1}$ \\
1 & 1\,000    & 1.314 & 1.358 & $3.2\times 10^{-2}$ & $8.3\times 10^{-2}$ \\
1 & 10\,000   & 1.517 & 1.358 & $1.0\times 10^{-2}$ & $3.0\times 10^{-2}$ \\
1 & $10^6$    & 1.858 & 1.358 & $1.0\times 10^{-3}$ & $3.7\times 10^{-3}$ \\
\midrule
2 & 100       & 1.517 & 1.358 & $1.0\times 10^{-2}$ & $4.6\times 10^{-2}$ \\
2 & 1\,000    & 1.858 & 1.358 & $1.0\times 10^{-3}$ & $6.9\times 10^{-3}$ \\
2 & 10\,000   & 2.146 & 1.358 & $1.0\times 10^{-4}$ & $9.2\times 10^{-4}$ \\
2 & $10^6$    & 2.628 & 1.358 & $1.0\times 10^{-6}$ & $1.4\times 10^{-5}$ \\
\midrule
5 & 100       & 2.399 & 1.358 & $1.0\times 10^{-5}$ & $4.6\times 10^{-4}$ \\
5 & 1\,000    & 2.938 & 1.358 & $3.2\times 10^{-8}$ & $4.0\times 10^{-6}$ \\
5 & 10\,000   & 3.393 & 1.358 & $1.0\times 10^{-10}$ & $2.6\times 10^{-8}$ \\
5 & $10^6$    & 4.156 & 1.358 & $1.0\times 10^{-15}$ & $7.1\times 10^{-13}$ \\
\midrule
10 & 100      & 3.393 & 1.358 & $1.0\times 10^{-10}$ & $2.1\times 10^{-7}$ \\
10 & 1\,000   & 4.156 & 1.358 & $1.0\times 10^{-15}$ & $1.6\times 10^{-11}$ \\
10 & 10\,000  & 4.799 & 1.358 & $1.0\times 10^{-20}$ & $6.6\times 10^{-16}$ \\
10 & $10^6$   & 5.877 & 1.358 & $1.0\times 10^{-30}$ & $5.0\times 10^{-25}$ \\
\bottomrule
\end{tabular}
\end{table}

For $\kappa\ge 5$ the MDP threshold already exceeds the fixed-$\alpha$ value at $n=100$; as $n$ grows, the gap widens steadily. The Bayes risk $B_n^{*}\asymp(\log n/n)^{\kappa/2}$ decays polynomially.

\subsection{\texorpdfstring{Multinomial $\chi^2$ Thresholds (Theorem~\ref{thm:chi-squared-threshold})}{Multinomial chi-squared Thresholds}}

With $\rhotail=1/4$ and $\kappa = k-1$ (Dirichlet prior), Table~\ref{tab:chisq} compares the MDP critical value $\chi^{2*}_n = (k-1)\log n$ against the conventional $\chi^2_{k-1}(0.95)$ quantile.

\begin{table}[!ht]
\centering
\caption{Multinomial $\chi^2$ Bayes-optimal thresholds ($\rhotail=1/4$).}\label{tab:chisq}
\smallskip
\begin{tabular}{rrrccr}
\toprule
$k$ & $\kappa$ & $n$ & $\chi^{2*}_n$ (MDP) & $\chi^2_{\kappa}(0.95)$ & $\alpha_n^{*}$ \\
\midrule
3 & 2 & 100     &  9.2 &  5.99 & $1.0\times 10^{-2}$ \\
3 & 2 & 1\,000  & 13.8 &  5.99 & $1.0\times 10^{-3}$ \\
3 & 2 & 10\,000 & 18.4 &  5.99 & $1.0\times 10^{-4}$ \\
\midrule
4 & 3 & 100     & 13.8 &  7.81 & $1.0\times 10^{-3}$ \\
4 & 3 & 1\,000  & 20.7 &  7.81 & $3.2\times 10^{-5}$ \\
4 & 3 & 10\,000 & 27.6 &  7.81 & $1.0\times 10^{-6}$ \\
\midrule
10 & 9 & 100     & 41.4 & 16.92 & $1.0\times 10^{-9}$ \\
10 & 9 & 1\,000  & 62.2 & 16.92 & $3.2\times 10^{-14}$ \\
10 & 9 & 10\,000 & 82.9 & 16.92 & $1.0\times 10^{-18}$ \\
\bottomrule
\end{tabular}
\end{table}

For $k=10$ categories at $n=10{,}000$, the Bayes-optimal critical value of $82.9$ is nearly $5\times$ the fixed-$\alpha$ value of $16.92$.
This reflects the multiplicity cost of servicing a $(k{-}1)$-dimensional alternative space under Bayes risk.

\subsection{Fisher Rejection Radii (Theorem~\ref{thm:fisher-mdp})}

With $\rhotail=1/4$ and $\kappa=\lambda+d$, Table~\ref{tab:fisher} shows the Fisher-distance rejection radius $r_n^{*}=\sqrt{(\lambda+d)\log n/n}$ for several configurations of prior exponent~$\lambda$ and parameter dimension~$d$.

\begin{table}[!ht]
\centering
\caption{Fisher-distance rejection radii ($\rhotail=1/4$, $\kappa=\lambda+d$).}\label{tab:fisher}
\smallskip
\begin{tabular}{rrrrcccc}
\toprule
$\lambda$ & $d$ & $\kappa$ & $n=100$ & $n=1{,}000$ & $n=10{,}000$ & $n=100{,}000$ \\
\midrule
1 & 1 & 2 & 0.3035 & 0.1175 & 0.0429 & 0.0152 \\
1 & 2 & 3 & 0.3717 & 0.1440 & 0.0526 & 0.0186 \\
2 & 3 & 5 & 0.4799 & 0.1858 & 0.0679 & 0.0240 \\
1 & 5 & 6 & 0.5257 & 0.2036 & 0.0743 & 0.0263 \\
\bottomrule
\end{tabular}
\end{table}

The radii shrink as $\sqrt{\log n/n}$, strictly slower than the $n^{-1/2}$ rate of CLT-scale procedures but strictly faster than the $O(1)$ resolution of LDP-scale tests.

\end{document}